\documentclass[a4paper,11pt,twoside]{amsart}

\usepackage[T1]{fontenc}
\usepackage[latin1]{inputenc}
\usepackage{xspace} 
\usepackage{amsfonts}
\usepackage{amsmath,amssymb,amsthm,amscd}
\usepackage[all]{xy}

\title[\tiny Deformation of proper actions on reductive homogeneous spaces]{Deformation of proper actions on reductive homogeneous spaces}
\author{Fanny Kassel}
\address{D\'epartement de Math\'ematiques,
B\^atiment~425,
Facult\'e des Sciences d'Orsay,
Universit\'e Paris-Sud 11,
91405 Orsay Cedex,
France}
\email{fanny.kassel@math.u-psud.fr}

\theoremstyle{plain}
\newtheorem{prop}{Proposition}[section]
\newtheorem{theo}[prop]{Theorem}
\newtheorem{coro}[prop]{Corollary}
\newtheorem{lem}[prop]{Lemma}
\theoremstyle{definition}
\newtheorem{rema}[prop]{Remark}

\newcommand{\C}{\mathbb{C}}
\newcommand{\R}{\mathbb{R}}
\newcommand{\Q}{\mathbb{Q}}
\newcommand{\Z}{\mathbb{Z}}
\newcommand{\N}{\mathbb{N}}
\newcommand{\F}{\mathbb{F}}
\newcommand{\kkk}{\mathbf{k}}
\newcommand{\GLnC}{\mathrm{GL}_n(\mathbb{C})}
\newcommand{\GLnR}{\mathrm{GL}_n(\mathbb{R})}

\newcommand{\SLn}{\mathrm{SL}_n}
\newcommand{\SL}{\mathrm{SL}}
\newcommand{\PGL}{\mathrm{PGL}}
\newcommand{\SO}{\mathrm{SO}}

\newcommand{\SU}{\mathrm{SU}}
\newcommand{\U}{\mathrm{U}}
\newcommand{\Sp}{\mathrm{Sp}}
\newcommand{\M}{\mathrm{M}}
\newcommand{\g}{\mathfrak{g}}
\newcommand{\p}{\mathfrak{p}}
\newcommand{\n}{\mathfrak{n}}
\newcommand{\aaa}{\mathfrak{a}}
\newcommand{\z}{\mathfrak{z}}
\newcommand{\so}{\mathfrak{so}}
\newcommand{\uu}{\mathfrak{u}}
\newcommand{\HH}{\mathbb{H}}
\newcommand{\PP}{\mathbb{P}}
\newcommand{\A}{\mathbb{A}}
\newcommand{\D}{\mathcal{D}}
\newcommand{\Hom}{\mathrm{Hom}}
\newcommand{\End}{\mathrm{End}}
\newcommand{\Int}{\mathrm{Int}}
\newcommand{\diag}{\mathrm{diag}}
\DeclareMathOperator{\Res}{Res}
\DeclareMathOperator{\Ad}{Ad}
\DeclareMathOperator{\ad}{ad}
\DeclareMathOperator{\pr}{pr}

\begin{document}
\maketitle
\numberwithin{equation}{section}
\numberwithin{table}{section}

\begin{abstract}
Let $G$ be a real reductive Lie group and $H$ a closed reductive subgroup of~$G$.
We investigate the deformation of ``standard'' compact quotients of~$G/H$, \textit{i.e.}, of quotients of~$G/H$ by discrete subgroups~$\Gamma$ of~$G$ that are uniform lattices in a closed reductive subgroup~$L$ of~$G$ acting properly and cocompactly on~$G/H$.
For $L$ of real rank~$1$, we prove that after a small deformation in~$G$, such a group~$\Gamma$ remains discrete in~$G$ and its action on~$G/H$ remains properly discontinuous and cocompact.
More generally, we prove that the properness of the action of any convex cocompact subgroup of~$L$ on~$G/H$ is preserved under small deformations, and we extend this result to reductive homogeneous spaces~$G/H$ over any local field.
As an application, we obtain compact quotients of $\SO(2n,2)/\U(n,1)$ by Zariski-dense discrete subgroups of~$\SO(2n,2)$ acting properly discontinuously.
\vspace{0.2cm}
\end{abstract}

\section{Introduction}

Let $G$ be a real connected reductive linear Lie group and $H$ a closed connected reductive subgroup of~$G$.
We are interested in the compact quotients of~$G/H$ by discrete subgroups~$\Gamma$ of~$G$.
We ask that the action of~$\Gamma$ on~$G/H$ be properly discontinuous in order for the quotient $\Gamma\backslash G/H$ to be Hausdorff.
This imposes strong restrictions on~$\Gamma$ when $H$ is noncompact.
For instance, if $\mathrm{rank}_{\R}(G)=\mathrm{rank}_{\R}(H)$, then all discrete subgroups of~$G$ acting properly discontinuously on~$G/H$ are finite: this is the Calabi-Markus phenomenon \cite{kob89}.
Usually the action of~$\Gamma$ on~$G/H$ is also required to be free, so that $\Gamma\backslash G/H$ be a manifold, but this condition is not very restrictive: if $\Gamma$ acts properly discontinuously and cocompactly on~$G/H$, then it is finitely generated, hence virtually torsion-free by Selberg's lemma~\cite{sel}.

In this paper we investigate the deformation of compact quotients $\Gamma\backslash G/H$ in the important case when $\Gamma$ is ``standard'', \textit{i.e.}, when $\Gamma$ is a uniform lattice in some closed reductive subgroup~$L$ of~$G$ acting properly and cocompactly on~$G/H$.
Most of our results hold for reductive homogeneous spaces over any local field, but in this introduction we first consider the real case.

\subsection{Deformation of compact quotients in the real case}

Let $G$ be a real reductive linear Lie group and $H$ a closed reductive subgroup of~$G$.
In all known examples, if $G/H$ admits a compact quotient, then there is a closed reductive subgroup~$L$ of~$G$ that acts properly and cocompactly on~$G/H$.
For instance, $L=\U(n,1)$ acts properly and transitively on the $(2n+1)$-dimensional anti-de Sitter space $G/H=\SO(2n,2)/\SO(2n,1)$ (see Section~\ref{quotients compacts Zariski-denses}).
Any torsion-free uniform lattice~$\Gamma$ of such a group~$L$ acts properly discontinuously, freely, and cocompactly on~$G/H$; we will say that the corresponding compact quotient $\Gamma\backslash G/H$ is \textit{standard}.
Note that $L$ always admits torsion-free uniform lattices by~\cite{bor63}.
Kobayashi and Yoshino conjectured that any reductive homogeneous space~$G/H$ admitting compact quotients admits standard ones (\cite{ky}, Conj.~3.3.10); this conjecture remains open.

Of course, nonstandard compact quotients may also exist: this is the case for instance for $G/H=(G_0\times G_0)/\Delta_{G_0}$, where $G_0$ is any reductive Lie group and $\Delta_{G_0}$ is the diagonal of $G_0\times G_0$ (see \cite{ghy}, \cite{gol}, \cite{kob98}, \cite{sal00}).
But in general we know only standard examples.
In order to construct nonstandard ones, it is natural, given a reductive subgroup~$L$ of~$G$ acting properly and cocompactly on~$G/H$, to slightly deform torsion-free uniform lattices~$\Gamma$ of~$L$ in~$G$ and to see whether they remain discrete in~$G$ and their action on~$G/H$ remains proper, free, and cocompact.

Since our goal is to obtain nonstandard quotients, we are not really interested in trivial deformations of~$\Gamma$, \textit{i.e.}, in deformations by conjugation, for which the quotient $\Gamma\backslash G/H$ will remain standard.
The problem therefore boils down to the case when $L$ is reductive of real rank~$1$.
Indeed, by classical preliminary reductions, that is, considering separately the different irreducible quasifactors of~$\Gamma$ and possibly passing to subgroups of finite index, we may assume that $\Gamma$ is irreducible and that $L$ has no compact factor.
If $L$ is semisimple of real rank~$\geq 2$, then Margulis's superrigidity theorem implies that $\Gamma$ is locally rigid in~$G$ (\cite{mar91}, Cor.~IX.5.9).

We prove the following result.

\begin{theo}\label{proprete, groupes de Lie}
Let $G$ be a real reductive linear Lie group, $H$ and~$L$ two closed reductive subgroups of~$G$.
Assume that $\mathrm{rank}_{\R}(L)=1$ and that $L$ acts properly and cocompactly on~$G/H$.
For any torsion-free uniform lattice~$\Gamma$ of~$L$, there is a neighborhood~$\mathcal{U}\subset\Hom(\Gamma,G)$ of the natural inclusion such that any $\varphi\in\mathcal{U}$ is injective and $\varphi(\Gamma)$ is discrete in~$G$, acting properly discontinuously and cocompactly on~$G/H$.
\end{theo}

We denote by $\Hom(\Gamma,G)$ the set of group homomorphisms from~$\Gamma$ to~$G$, endowed with the compact-open topology.
In the real case, the fact that $\varphi(\Gamma)$ remains discrete in~$G$ for $\varphi\in\Hom(\Gamma,G)$ close to the natural inclusion is a general result of Guichard (\cite{gui}, Th.~2).

Theorem~\ref{proprete, groupes de Lie} improves a result of Kobayashi (\cite{kob98}, Th.~2.4), who considered particular homomorphisms of the form $\gamma\mapsto\gamma\psi(\gamma)$, where $\psi : \Gamma\rightarrow\nolinebreak Z_G(L)$ is a homomorphism with values in the centralizer of~$L$ in~$G$.

By~\cite{ky}, Cor.~3.3.7, Theorem~\ref{proprete, groupes de Lie} applies to the following triples~$(G,H,L)$:
\begin{enumerate}
	\item $(\SO(2n,2),\SO(2n,1),\U(n,1))$ for~$n\geq 1$,
	\item $(\SO(2n,2),\U(n,1),\SO(2n,1))$ for~$n\geq 1$,
	\item $(\U(2n,2),\Sp(n,1),\U(2n,1))$ for~$n\geq 1$,
	\item $(\SO(8,8),\SO(8,7),\mathrm{Spin}(8,1))$,
	\item $(\SO(8,\C),\SO(7,\C),\mathrm{Spin}(7,1))$,
	\item $(\SO^{\ast}(8),\U(1,3),\mathrm{Spin}(6,1))$,
	\item $(\SO^{\ast}(8),\mathrm{Spin}(6,1),\U(1,3))$,
	\item $(\SO^{\ast}(8),\SO^{\ast}(6)\times\SO^{\ast}(2),\mathrm{Spin}(6,1))$,
	\item $(\SO(4,4),\mathrm{Spin}(4,3),\SO(4,1))$,
	\item $(\SO(4,3),\mathrm{G}_{2(2)},\SO(4,1))$.
\end{enumerate}

\medskip

As mentioned above, our aim is to deform standard compact quotients of~$G/H$ into nonstandard ones, which are in some sense more generic.
The best that we may hope for is to obtain Zariski-dense discrete subgroups of~$G$ acting properly discontinuously, freely, and cocompactly on~$G/H$.
Of course, even when $L$ has real rank~$1$, nontrivial deformations in~$G$ of uniform lattices~$\Gamma$ of~$L$ do not always exist.
For instance, if $L$ is semisimple, noncompact, with no quasisimple factor locally isomorphic to~$\SO(n,1)$ or~$\SU(n,1)$, then the first cohomology group~$H^1(\Gamma,\g)$ vanishes by~\cite{rag}, Th.~1.
This, together with~\cite{wei64}, implies that $\Gamma$ is locally rigid in~$G$.
(Here $\g$ denotes the Lie algebra of~$G$.)

For $(G,H,L)=(\SO(2n,2),\SO(2n,1),\U(n,1))$ with $n\geq 2$, uniform lattices~$\Gamma$ of~$L$ are not locally rigid in~$G$, but a small deformation of~$\Gamma$ will never provide a Zariski-dense subgroup of~$G$.
Indeed, by \cite{rag} and~\cite{wei64} there is a neighborhood in~$\Hom(\Gamma,G)$ of the natural inclusion whose elements are all homomorphisms of the form $\gamma\mapsto\nolinebreak\gamma\psi(\gamma)$, where $\psi : \Gamma\rightarrow\SO(2n,2)$ is a homomorphism with values in the center of~$\U(n,1)$.

On the other hand, for $(G,H,L)=(\SO(2n,2),\U(n,1),\SO(2n,1))$ with $n\geq 1$, there do exist small Zariski-dense deformations of certain uniform lattices of~$L$ in~$G$ (see Section~\ref{quotients compacts Zariski-denses}): such deformations can be obtained by a bending construction due to Johnson and Millson~\cite{jm}.
Theorem~\ref{proprete, groupes de Lie} therefore implies the following result on the compact quotients of the homogeneous space $G/H=\SO(2n,2)/\U(n,1)$.

\begin{coro}\label{quotients compacts de SO(2n,2)/SU(n,1)}
For any~$n\geq 1$, there is a Zariski-dense discrete subgroup of~$\SO(2n,2)$ acting properly discontinuously, freely, and cocompactly on \linebreak $\SO(2n,2)/\U(n,1)$.
\end{coro}

Note that by~\cite{ky}, Prop.~3.2.7, the homogeneous space $\SO(2n,2)/\U(n,1)$ is a pseudo-Riemannian symmetric space of signature~$(2n,n^2-1)$.

The existence of compact quotients of reductive homogeneous spaces by Zariski-dense discrete subgroups was known so far only for homogeneous spaces of the form $(G_0\times G_0)/\Delta_{G_0}$.

\subsection{Deformation of properly discontinuous actions over a general local field}

We prove that the properness of the action is preserved under small deformations not only for real groups, but more generally for algebraic groups over any local field~$\kkk$.
By a local field we mean $\R$, $\C$, a finite extension of~$\Q_p$, or the field $\F_q((t))$ of formal Laurent series over a finite field~$\F_q$.
Moreover we relax the assumption that $\Gamma$ is a torsion-free uniform lattice of~$L$, in the following way.

\begin{theo}\label{proprete, groupes algebriques}
Let $\kkk$ be a local field, $G$ the set of $\kkk$-points of a reductive algebraic $\kkk$-group~$\mathbf{G}$, and $H$ (resp.~$L$) the set of $\kkk$-points of a closed reductive subgroup~$\mathbf{H}$ (resp.~$\mathbf{L}$) of~$\mathbf{G}$.
Assume that $\mathrm{rank}_{\kkk}(\mathbf{L})=1$ and that $L$ acts properly on~$G/H$.
If $\kkk=\R$ or~$\C$, let $\Gamma$ be a torsion-free convex cocompact subgroup of~$L$; if $\kkk$ is non-Archimedean, let $\Gamma$ be any torsion-free finitely generated discrete subgroup of~$L$.
Then there is a neighborhood $\mathcal{U}\subset\Hom(\Gamma,G)$ of the natural inclusion such that any $\varphi\in\mathcal{U}$ is injective and $\varphi(\Gamma)$ is discrete in~$G$, acting properly discontinuously on~$G/H$.
\end{theo}

Recall that for $\kkk=\R$ or~$\C$, a discrete subgroup of~$L$ is called \textit{convex cocompact} if it acts cocompactly on the convex hull of its limit set in the symmetric space of~$L$.
In particular, any uniform lattice of~$L$ is convex cocompact.

For $\kkk=\R$ or~$\C$, Theorem~\ref{proprete, groupes de Lie} follows from Theorem~\ref{proprete, groupes algebriques} and from a cohomological argument due to Kobayashi (see Section~\ref{Proprete and deformations}).
This argument does not transpose to the non-Archimedean case.

Note that in characteristic zero, every finitely generated subgroup of~$L$ is virtually torsion-free by Selberg's lemma (\cite{sel}, Lem.~8), hence the ``torsion-free'' assumption in Theorem~\ref{proprete, groupes algebriques} may easily be removed in this case.

\subsection{Translation in terms of a Cartan projection}

Let $\kkk$ be a local field and $G$ the set of $\kkk$-points of a connected reductive algebraic $\kkk$-group.
Fix a Cartan projection $\mu : G\rightarrow E^+$ of~$G$, where $E^+$ is a closed convex cone in a real finite-dimensional vector space~$E$ (see Section~\ref{Preliminaires}).
For any closed subgroup~$H$ of~$G$, the \textit{properness criterion} of Benoist (\cite{ben96}, Cor.~5.2) and Kobayashi (\cite{kob96}, Th.~1.1) translates the properness of the action on~$G/H$ of a subgroup~$\Gamma$ of~$G$ in terms of~$\mu$.
Using this criterion (see Subsection~\ref{Proprete and deformations}), Theorem~\ref{proprete, groupes algebriques} is a consequence of the following result, where we fix a norm~$\Vert\cdot\Vert$ on~$E$.

\begin{theo}\label{varphi ne change pas beaucoup mu}
Let $\kkk$ be a local field, $G$ the set of $\kkk$-points of a connected reductive algebraic $\kkk$-group~$\mathbf{G}$, and $L$ the set of $\kkk$-points of a closed reductive subgroup~$\mathbf{L}$ of~$\mathbf{G}$ of $\kkk$-rank~$1$.
If $\kkk=\R$ or~$\C$, let $\Gamma$ be a convex cocompact subgroup of~$L$; if $\kkk$ is non-Archimedean, let $\Gamma$ be any finitely generated discrete subgroup of~$L$.
For any $\varepsilon>0$, there is a neighborhood $\mathcal{U}_{\varepsilon}\subset\Hom(\Gamma,G)$ of the natural inclusion and a constant $C_{\varepsilon}\geq 0$ such that
$$\big\Vert\mu(\varphi(\gamma)) - \mu(\gamma)\big\Vert \leq \varepsilon \Vert\mu(\gamma)\Vert + C_{\varepsilon}$$
for all $\varphi\in\mathcal{U}_{\varepsilon}$ and all~$\gamma\in\Gamma$.
\end{theo}

\subsection{Ideas of proofs}

The core of the paper is the proof of Theorem~\ref{varphi ne change pas beaucoup mu}.
We start by recalling, in Section~\ref{Preliminaires}, that certain linear forms~$\ell$ on~$E$ are connected to representations~$(V,\rho)$ of~$\mathbf{G}$ by relations of the form
$$\ell(\mu(g)) = \log\Vert\rho(g)\Vert_{_V}$$
for all~$g\in G$, where~$\Vert\cdot\Vert_{_V}$ is a certain fixed norm on~$V$.
We are thus led to bound ratios of the form $\Vert\rho(\varphi(\gamma))\Vert_{_V}/\Vert\rho(\gamma)\Vert_{_V}$, where~$\gamma\in\Gamma\smallsetminus\{ 1\} $ and where $\varphi\in\Hom(\Gamma,G)$ is close to the natural inclusion of~$\Gamma$ in~$G$.

In order to bound these ratios we look at the dynamics of~$G$ acting on the projective space~$\PP(V)$, notably the dynamics of the elements~$g\in G$ that are \textit{proximal} in~$\PP(V)$.
By definition, such elements~$g\in G$ admit an attracting fixed point and a repelling projective hyperplane in~$\PP(V)$.
In Section~\ref{Dynamique proximale} we consider products $z_1k_2z_2\ldots k_nz_n$ of proximal elements~$z_i$ having a common attracting fixed point~$x_0^+$ and a common repelling hyperplane~$X_0^-$, with isometries~$k_i$ such that $k_i\cdot x_0^+$ remains bounded away from~$X_0^-$.
We estimate the contraction power of such a product in terms of the contraction powers of the~$z_i$.

In Section~\ref{Partie produit transverse} we see how such dynamical considerations apply to the elements $\gamma\in\Gamma$ and their images~$\varphi(\gamma)$ under a small deformation $\varphi\in\nolinebreak\Hom(\Gamma,G)$.
We use Guichard's idea~\cite{gui} of writing every element~$\gamma\in\Gamma$ as a product $\gamma_0\ldots\gamma_n$ of elements of a fixed finite subset~$F$ of~$\Gamma$, where the norms~$\Vert\mu(\gamma_i)\Vert$ and $\Vert\mu(\gamma_i\gamma_{i+1})-\mu(\gamma_i)-\mu(\gamma_{i+1})\Vert$ are controlled for all~$i$.

In Section~\ref{Demonstration des theoremes} we combine the results of Sections~\ref{Dynamique proximale} and~\ref{Partie produit transverse} by carefully choosing the finite subset~$F$ of~$\Gamma$ in order to get a sharp control of the ratios~$\Vert\rho(\varphi(\gamma))\Vert_{_V}/\Vert\rho(\gamma)\Vert_{_V}$, or equivalently of $\ell(\mu(\varphi(\gamma))-\mu(\gamma))$ for~$\gamma\in\Gamma\smallsetminus\{ 1\} $.
From this we deduce Theorem~\ref{varphi ne change pas beaucoup mu}.

At the end of Section~\ref{Demonstration des theoremes} we explain how Theorems~\ref{proprete, groupes de Lie} and~\ref{proprete, groupes algebriques} follow from Theorem~\ref{varphi ne change pas beaucoup mu}.
Finally, in Section~\ref{quotients compacts Zariski-denses} we establish Corollary~\ref{quotients compacts de SO(2n,2)/SU(n,1)} by relating Theorem~\ref{proprete, groupes de Lie} to Johnson and Millson's bending construction.

\subsection*{Acknowledgements}

I warmly thank Yves Benoist and Olivier Guichard for fruitful discussion.

\section{Cartan projections, maximal parabolic subgroups, and representations}\label{Preliminaires}

Throughout the paper, $\kkk$ denotes a local field, \textit{i.e.}, $\R$, $\C$, a finite extension of~$\Q_p$, or the field $\F_q((t))$ of formal Laurent series over a finite field~$\F_q$.
If $\kkk=\R$ or~$\C$, we denote by~$|\cdot|$ the usual absolute value on~$\kkk$.
If $\kkk$ is non-Archimedean, we denote by~$\mathcal{O}$ the ring of integers of~$\kkk$, by~$q$ the cardinal of its residue field, by~$\pi$ a uniformizer, by~$\omega$ the (additive) valuation on~$\kkk$ such that $\omega(\pi)=1$, and by $|\cdot| = q^{-\omega(\cdot)}$ the corresponding (multiplicative) absolute value.
If $\mathbf{G}$ is an algebraic group, we denote by~$G$ the set of its $\kkk$-points and by~$\g$ its Lie algebra.

In this section, we recall a few well-known facts on connected reductive algebraic $\kkk$-groups and their Cartan projections.

\subsection{Weyl chambers}\label{chambre de Weyl}

Fix a connected reductive algebraic $\kkk$-group $\mathbf{G}$.
The derived group~$\mathbf{D(G)}$ is semisimple, the identity component~$\mathbf{Z(G)}^{\circ}$ of the center of~$\mathbf{G}$ is a torus, which is trivial if $\mathbf{G}$ is semisimple, and $\mathbf{G}$ is the almost product of~$\mathbf{D(G)}$ and~$\mathbf{Z(G)}^{\circ}$.
Recall that the $\kkk$-split $\kkk$-tori of~$\mathbf{G}$ are all conjugate over~$\kkk$.
Fix such a torus~$\mathbf{A}$ and let $\mathbf{N}$ (resp.~$\mathbf{Z}$) denote its normalizer (resp.\ centralizer) in~$\mathbf{G}$.
The group $X(\mathbf{A})$ of $\kkk$-characters of~$\mathbf{A}$ and the group $Y(\mathbf{A})$ of $\kkk$-cocharacters are both free $\Z$-modules of rank $\mathrm{rank}_{\kkk}(\mathbf{G})$ and there is a perfect pairing
$$\langle\cdot\,,\cdot\rangle : X(\mathbf{A})\times Y(\mathbf{A})\longrightarrow\Z.$$
Note that $\mathbf{A}$ is the almost product of $(\mathbf{A}\cap\mathbf{D(G)})^{\circ}$ and~$(\mathbf{A}\cap\mathbf{Z(G)})^{\circ}$, hence $X(\mathbf{A})\otimes_{\Z}\R$ is the direct sum of $X((\mathbf{A}\cap\mathbf{D(G)})^{\circ})\otimes_{\Z}\R$ and $X((\mathbf{A}\cap\nolinebreak\mathbf{Z(G)})^{\circ})\otimes_{\Z}\nolinebreak\R$.

The set $\Phi=\Phi(\mathbf{A},\mathbf{G})$ of restricted roots of~$\mathbf{A}$ in~$\mathbf{G}$, \textit{i.e.}, the set of nontrivial weights of~$\mathbf{A}$ in the adjoint representation of~$\mathbf{G}$, is a root system of $X((\mathbf{A}\cap\mathbf{D(G)})^{\circ})\otimes_{\Z}\R$.
For~$\alpha\in\Phi$, let~$\check{\alpha}$ be the corresponding coroot: by definition, $\langle\alpha,\check{\alpha}\rangle=2$ and $s_{\alpha}(\Phi)=\Phi$, where $s_{\alpha}$ is the reflection of $X(\mathbf{A})\otimes_{\Z}\R$ mapping~$x$ to $x - \langle x,\check{\alpha}\rangle\,\alpha$.
The group $W=N/Z$ is finite and identifies with the Weyl group of~$\Phi$, generated by the reflections~$s_{\alpha}$.

Similarly, $E=Y(\mathbf{A})\otimes_{\Z}\R$ is the direct sum of $E_D=Y((\mathbf{A}\cap\mathbf{D(G)})^{\circ})\otimes_{\Z}\R$ and $E_Z=Y((\mathbf{A}\cap\nolinebreak\mathbf{Z(G)})^{\circ})\otimes_{\Z}\R$.
The group $W=N/Z$ acts trivially on~$E_Z$ and identifies with the Weyl group of the root system $\check{\Phi}=\{ \check{\alpha},\ \alpha\in\Phi\} $ of~$E_D$.
We refer to~\cite{bot} for proofs and more detail.

If $\kkk$ is non-Archimedean, set $A^{\circ}=A$; if $\kkk=\R$ or~$\C$, set
$$A^{\circ} = \big\{ a\in A\,,\quad \chi(a)\in\, ]0,+\infty[ \quad\forall\chi\in X(\mathbf{A})\big\} .$$
Choose a basis~$\Delta$ of~$\Phi$ and let
$$\begin{array}{lcclcl}
& A^+ & = & \big\{ a\in A^{\circ}\!, & \ |\alpha(a)|\geq 1 & \forall\alpha\in\Delta\big\} \\
\mathrm{\big(resp.}\quad & E^+ & = & \big\{ x\in E, & \langle\alpha,x\rangle\geq 0 & \forall\alpha\in\Delta\big\} \mathrm{\big)}
\end{array}$$
denote the corresponding closed positive Weyl chamber in~$A^{\circ}$ (resp.\ in~$E$).
The set~$E^+$ is a closed convex cone in the real vector space~$E$.
If $\kkk=\R$ or~$\C$, then $E$ identifies with~$\aaa$ and $E^+$ with $\log A^+\subset\aaa$, and we endow $E$ with the Euclidean norm~$\Vert\cdot\Vert$ induced by the Killing form of~$\g$.
If $\kkk$ is non-Archimedean, we endow~$E$ with any $W$-invariant Euclidean norm~$\Vert\cdot\Vert$.

\subsection{Cartan decompositions and Cartan projections}\label{Projection de Cartan}

If $\kkk=\R$ or~$\C$, there is a maximal compact subgroup $K$ of~$G$ such that the Cartan decomposition $G=KA^+K$ holds: for $g\in G$, there are elements $k_g,\ell_g\in K$ and a unique $a_g\in A^+$ such that $g = k_g a_g \ell_g$ (\cite{hel}, Chap.~9, Th.~1.1).
Setting $\mu(g)=\log a_g$ defines a map $\mu : G\rightarrow E^+\simeq\log A^+$, which is continuous, proper, and surjective.
It is called the \emph{Cartan projection} with respect to the Cartan decomposition $G=KA^+K$.

If $\kkk$ is non-Archimedean, let $\Res : X(\mathbf{Z})\rightarrow X(\mathbf{A})$ denote the restriction homomorphism, where $X(\mathbf{Z})$ is the group of $\kkk$-characters of~$\mathbf{Z}$.
There is a unique group homomorphism $\nu : Z\rightarrow E$ such that
$$\langle\Res(\chi),\nu(z)\rangle = -\,\omega(\chi(z))$$
for all $\chi\in X(\mathbf{Z})$ and~$z\in Z$.
Let $Z^+\subset Z$ denote the inverse image of~$E^+$ under~$\nu$.
The \emph{Cartan decomposition} $G = KZ^+K$ holds: for $g\in G$, there are elements $k_g,\ell_g\in K$ and $z_g\in Z^+$ such that $g = k_g z_g \ell_g$, and $\nu(z_g)$ is uniquely defined.
Setting $\mu(g)=\nu(z_g)$ defines a map $\mu : G\rightarrow E^+$, which is continuous and proper, and whose image $\mu(G)$ is the intersection of~$E^+$ with a lattice of~$E$.
It is called the \emph{Cartan projection} with respect to the Cartan decomposition~$G=KZ^+K$.
For proofs and more detail we refer to the original articles \cite{bt1} and~\cite{bt2}, but the reader may also find~\cite{rou} a useful reference.

\subsection{A geometric interpretation}

If $\kkk=\R$ or~$\C$, let $X=G/K$ denote the Riemannian symmetric space of~$G$ and $d$ its distance; set $x_0=K\in X$.
Since $G$ acts on~$X$ by isometries, we have
\begin{equation}\label{prelim mu distance, reel}
\Vert\mu(g)\Vert = d(x_0,g\cdot x_0)
\end{equation}
for all~$g\in G$.
If $\kkk$ is non-Archimedean, let $X$ denote the \emph{Bruhat-Tits building} of~$G$: it is a metric space on which $G$ acts properly by isometries with a compact fundamental domain (see \cite{bt1} or~\cite{rou}).
When $\mathrm{rank}_{\kkk}(\mathbf{G})=1$, it is a bipartite simplicial tree (see \cite{ser77}, \S~II.1, for the case of $G=\SL_2(\kkk)$).
The group~$K$ is the stabilizer of some point~$x_0\in X$, and we have
\begin{equation}\label{prelim mu distance, ultrametrique}
\Vert\mu(g)\Vert = d(x_0,g\cdot x_0)
\end{equation}
for all~$g\in G$, where $d$ denotes the distance on~$X$.
It follows from (\ref{prelim mu distance, reel}) and (\ref{prelim mu distance, ultrametrique}) that in both cases (Archimedean or not),
\begin{equation}\label{inegalite triangulaire pour mu}
\Vert\mu(gg')\Vert \,\leq\, \Vert\mu(g)\Vert + \Vert\mu(g')\Vert
\end{equation}
for all $g,g'\in G$.
In fact, the following stronger inequalities hold (see for instance \cite{kas08}, Lem.~2.3): for all $g,g'\in G$,
\begin{equation}\label{inegalite fine pour mu}
\left \{
\begin{array}{c @{\ \leq\ } c}
    \Vert\mu(gg')-\mu(g')\Vert & \Vert\mu(g)\Vert,\\
    \Vert\mu(gg')-\mu(g)\Vert & \Vert\mu(g')\Vert.
\end{array}
\right.
\end{equation}

\subsection{Maximal parabolic subgroups}\label{Sous-groupes paraboliques maximaux}

For~$\alpha\in\Phi$, let~$\mathbf{U}_{\alpha}$ denote the corresponding unipotent subgroup of~$\mathbf{G}$, with Lie algebra $\uu_{\alpha}=\g_{\alpha}\oplus\g_{2\alpha}$, where
$$\g_{i\alpha} = \big\{ X\in\g,\quad \Ad(a)(X)=\alpha(a)^iX\quad \forall a\in A\big\} $$
for~$i=1,2$.
For any subset~$\theta$ of~$\Delta$, let~$\mathbf{P}_{\theta}$ denote the corresponding \textit{standard} parabolic subgroup of~$\mathbf{G}$, with Lie algebra
$$\p_{\theta} = \z \oplus \Big(\bigoplus_{\beta\in\Phi^+} \uu_{\beta}\Big) \oplus \Big(\bigoplus_{\beta\in\N(\Delta\smallsetminus\theta)} \uu_{-\beta}\Big).$$
Every parabolic $\kkk$-subgroup~$\mathbf{P}$ of~$\mathbf{G}$ is conjugate over~$\kkk$ to a unique standard one.
In particular, the maximal proper parabolic $\kkk$-subgroups of~$\mathbf{G}$ are the conjugates of the groups $\mathbf{P}_{\alpha}=\mathbf{P}_{\{ \alpha\} }$, where $\alpha\in\Delta$.

Fix~$\alpha\in\Delta$.
Since $\mathbf{P}_{\alpha}$ is its own normalizer in~$\mathbf{G}$, the \textit{flag variety} $\mathbf{G}/\mathbf{P}_{\alpha}$ parametrizes the set of parabolic $\kkk$-subgroups that are conjugate to~$\mathbf{P}_{\alpha}$.
It is a projective variety, which is defined over~$\kkk$.
Let~$\mathbf{N}_{\alpha}^-$ denote the unipotent subgroup of~$\mathbf{G}$ generated by the groups~$\mathbf{U}_{-\beta}$ for $\beta\in\alpha +\Phi^+$, with Lie algebra
$$\n_{\alpha}^- = \bigoplus_{\beta\in\Phi^+} \uu_{-(\alpha+\beta)}.$$
Let~$W_{\alpha}$ be the subgroup of~$W$ generated by the reflections~$s_{\beta}$ for $\beta\in\Delta\smallsetminus\{ \alpha\} $.
The \textit{Bruhat decomposition}
$$\mathbf{G}/\mathbf{P}_{\alpha} = \coprod_{w\in W/W_{\alpha}} \mathbf{N}_{\alpha}^- w \mathbf{P}_{\alpha}$$
holds, where the projective subvariety $\mathbf{N}_{\alpha}^- w \mathbf{P}_{\alpha}$ has positive codimension whenever~$wW_{\alpha}\neq W_{\alpha}$.
We refer to~\cite{bot} for proofs and more detail.

\subsection{Representations of~$\mathbf{G}$}\label{Representations de G}

For~$\alpha\in\Delta$, let $\omega_{\alpha}\in X(\mathbf{A})$ denote the corresponding fundamental weight: by definition, $\langle\omega_{\alpha},\check{\alpha}\rangle=1$ and $\langle\omega_{\alpha},\check{\beta}\rangle=\nolinebreak 0$ for all~$\beta\in\Delta\smallsetminus\{ \alpha\} $.
By \cite{tit71}, Th.~7.2, there is an irreducible $\kkk$-representation $(\rho_{\alpha},V_{\alpha})$ of~$\mathbf{G}$ whose highest weight~$\chi_{\alpha}$ is a positive multiple of~$\omega_{\alpha}$ and whose highest weight space~$x_{\alpha}^+$ is a line.
The point $x_{\alpha}^+\in\PP(V_{\alpha})$ is the unique fixed point of~$P_{\alpha}$ in~$\PP(V_{\alpha})$.
The map from $\mathbf{G}/\mathbf{P}_{\alpha}$ to~$\PP(V_{\alpha})$ sending~$g\mathbf{P}_{\alpha}$ to~$\rho_{\alpha}(g)(x_{\alpha}^+)$ is a closed immersion.
We denote the set of restricted roots of~$(\rho_{\alpha},V_{\alpha})$ by~$\Lambda_{\alpha}$ and, for every~$\lambda\in\Lambda_{\alpha}$, the weight space of~$\lambda$ by~$(V_{\alpha})_{\lambda}$.

If~$\kkk=\R$ (resp.\ if~$\kkk=\C$), then the weight spaces are orthogonal with respect to some $K$-invariant Euclidean (resp.\ Hermitian) norm~$\Vert\cdot\Vert_{\alpha}$ on~$V_{\alpha}$.
The corresponding operator norm~$\Vert\cdot\Vert_{\alpha}$ on~$\End(V_{\alpha})$ satisfies
\begin{equation}\label{norme des representations and Cartan projection, cas reel}
\Vert\rho_{\alpha}(g)\Vert_{\alpha} = e^{\langle\chi_{\alpha},\mu(g)\rangle}
\end{equation}
for all~$g\in G$.
If $\kkk$ is non-Archimedean, then there is a $K$-invariant ultrametric norm~$\Vert\cdot\Vert_{\alpha}$ on~$V_{\alpha}$ such that
$$\bigg\Vert\sum_{\lambda\in\Lambda_{\alpha}} v_{\lambda}\bigg\Vert_{\alpha} = \max_{\lambda\in\Lambda_{\alpha}} \Vert v_{\lambda}\Vert_{\alpha}$$
for all $(v_{\lambda})\in \prod_{\lambda\in\Lambda_{\alpha}}(V_{\alpha})_{\lambda}$ and such that the restriction of~$\rho_{\alpha}(z)$ to~$(V_{\alpha})_{\lambda}$ is a homothety of ratio~$q^{\langle\lambda,\nu(z)\rangle}$ for all $z\in Z$ and all $\lambda\in\Lambda_{\alpha}$ (\cite{qui}, Th.~6.1).
The corresponding operator norm~$\Vert\cdot\Vert_{\alpha}$ on~$\End(V_{\alpha})$ satisfies
\begin{equation}\label{norme des representations and Cartan projection, cas ultrametrique}
\Vert\rho_{\alpha}(g)\Vert_{\alpha} = q^{\langle\chi_{\alpha},\mu(g)\rangle}
\end{equation}
for all~$g\in G$.

\subsection{The example of $\mathbf{SL}_n$}

Let $\mathbf{G}=\mathbf{SL}_n$ for some integer $n\geq 2$.
The group $\mathbf{A}$ of diagonal matrices with determinant~$1$ is a maximal $\kkk$-split \linebreak $\kkk$-torus of~$\mathbf{G}$ which is its own centralizer, \textit{i.e.}, $\mathbf{Z}=\mathbf{A}$.
The corresponding \linebreak root system~$\Phi$ is the set of linear forms $\varepsilon_i-\varepsilon_j$, $1\leq i\neq j\leq n$, where
$$\varepsilon_i\big(\diag(a_1,\ldots,a_n)\big) = a_i.$$
The roots $\varepsilon_i-\varepsilon_{i+1}$, for $1\leq i\leq n-1$, form a basis~$\Delta$ of~$\Phi$.
If $\kkk$ is Archimedean (resp.\ non-Archimedean), the corresponding positive Weyl chamber is
\begin{eqnarray*}
A^+ & = & \big\{ \diag(a_1,\ldots,a_n)\in A,\ \ \! a_i\in\, ]0,+\infty[\ \forall i\ \,\mathrm{and}\ a_1\geq\ldots\geq a_n\big\} \\
\mathrm{\big(resp.}\quad A^+ & = & \big\{ \diag(a_1,\ldots,a_n)\in A,\ |a_1|\geq\ldots\geq|a_n|\big\} \mathrm{\big).}
\end{eqnarray*}
Set $K=\SO(n)$ (resp.\ $K=\SU(n)$, resp.\ $K=\SLn(\mathcal{O})$) if $\kkk=\R$ (resp.\ if $\kkk=\nolinebreak\C$, resp.\ if $\kkk$ is non-Archimedean).
The Cartan decomposition $G=KA^+K$ holds.
If $\kkk=\R$ (resp.\ if $\kkk=\nolinebreak\C$) it follows from the polar decomposition in $\GLnR$ (resp.\ in $\GLnC$) and from the reduction of symmetric (resp.\ Hermitian) matrices.
If $\kkk$ is non-Archimedean, it follows from the structure theorem for finitely generated modules over a principal ideal domain.
The real vector space
$$E = \big\{ (x_1,\ldots,x_n)\in\R^n,\ x_1+\ldots+x_n=0\big\}\ \simeq\ \R^{n-1}$$
and its closed convex cone
$$E^+ = \big\{ (x_1,\ldots,x_n)\in E,\ x_1\geq\ldots\geq x_n\big\} $$
do not depend on~$\kkk$.
Let $\mu : G\rightarrow E^+$ denote the Cartan projection with respect to the Cartan decomposition~$G=KA^+K$.
If $\kkk=\R$ or~$\C$, then $\mu(g)=(\frac{1}{2}\log x_i)_{1\leq i\leq n}$ where $x_i$ is the $i$-th eigenvalue of~$^t\!\overline{g}g$.
If $\kkk$ is non-Archimedean and if $m$ is any integer such that $\pi^mg\in\M_n(\mathcal{O})$, then $\mu(g)=(\omega(x_{m,i})-m)_{1\leq i\leq n}$ where $x_{m,i}$ is the $i$-th invariant factor of~$\pi^mg$.

Fix a simple root $\alpha=\varepsilon_{i_0}-\varepsilon_{i_0+1}\in\Delta$.
The parabolic group~$\mathbf{P}_{\alpha}$ is defined by the vanishing of the $(i,j)$-matrix entries for $1\leq j\leq i_0<i\leq n$.
The flag variety~$\mathbf{G}/\mathbf{P}_{\alpha}$ is the Grassmannian~$\mathcal{G}(i_0,n)$ of $i_0$-dimensional subspaces of the affine space~$\mathbb{A}^n$.
The Lie algebra~$\n_{\alpha}^-$ is defined by the vanishing of the $(i,j)$-matrix entries for $1\leq i\leq i_0$ and for $i_0+1\leq i,j\leq n$.
The decomposition
$$\mathbf{G}/\mathbf{P}_{\alpha} = \coprod_{w\in W/W_{\alpha}} \mathbf{N}_{\alpha}^- w \mathbf{P}_{\alpha}$$
is the decomposition of the Grassmannian~$\mathcal{G}(i_0,n)$ into Schubert cells.
The representation~$(\rho_{\alpha},V_{\alpha})$ is the natural representation of~$\mathbf{SL}_n$ in the wedge product~$\Lambda^{i_0}\A^n$.
Its highest weight is the fundamental weight
$$\omega_{\alpha}=\varepsilon_1+\ldots+\varepsilon_{i_0}$$
associated with~$\alpha$.
The embedding of the Grassmannian~$\mathcal{G}(i_0,n)$ into the projective space $\PP(V_{\alpha})=\PP(\Lambda^{i_0}\A^n)$ is the Plücker embedding.

\section{Dynamics in projective spaces}\label{Dynamique proximale}

In this section we look at the dynamics of certain endomorphisms of \linebreak $\kkk$-vector spaces in the corresponding projective spaces, where $\kkk$ is a local field.
In Subsection~\ref{Proximalite et normes} we start by recalling the notion of proximality.
We then consider products of the form $z_1k_2z_2\ldots k_nz_n$, where the~$z_i$ are proximal elements with a common attracting fixed point~$x_0^+$ and a common repelling hyperplane~$X_0^-$, and the~$k_i$ are isometries such that $k_i\cdot x_0^+$ remains bounded away from~$X_0^-$.
We estimate the contraction power of such a product in terms of the contraction powers of the~$z_i$.
In Subsection~\ref{Projection de Cartan and poids fondamentaux} we consider a connected reductive algebraic $\kkk$-group~$\mathbf{G}$ and apply the result of Subsection~\ref{Proximalite et normes} to the representations~$(V_{\alpha},\rho_{\alpha})$ of~$\mathbf{G}$ introduced in Subsection~\ref{Representations de G}.
From (\ref{norme des representations and Cartan projection, cas reel}) and~(\ref{norme des representations and Cartan projection, cas ultrametrique}) we get an upper bound for $|\langle\chi_{\alpha},\mu(g_1\ldots g_n)-\mu(g_1)-\ldots-\mu(g_n)\rangle|$ for elements $g_1,\ldots,g_n\in G$ satisfying certain contractivity and transversality conditions.

\subsection{Proximality in projective spaces and norm estimates}\label{Proximalite et normes}

Let $\kkk$ be a local field and $V$ be a finite-dimensional vector space over~$\kkk$.
Given a basis~$(v_1,\ldots,v_n)$ of~$V$, we define the norm
\begin{equation}\label{definition norme avec une base}
\bigg\Vert\sum_{1\leq j\leq n} t_j\,v_j\bigg\Vert_{_V} = \sup_{1\leq j\leq n} |t_j|
\end{equation}
on~$V$, and we keep the notation~$\Vert\cdot\Vert_{_V}$ for the corresponding operator norm on~$\End(V)$.
We endow the projective space~$\PP(V)$ with the distance
$$d(x_1,x_2) = \inf\big\{ \Vert v_1-v_2\Vert_{_V},\ v_i\in x_i\ \mathrm{and}\ \Vert v_i\Vert_{_V} =1\ \forall i=1,2\big\} .$$
Recall that an element $g\in\End(V)\smallsetminus\{ 0\} $ is called \textit{proximal} if it has a unique eigenvalue of maximal absolute value and if this eigenvalue has multiplicity~1.
(The eigenvalues of~$g$ belong to a finite extension~$\kkk_g$ of~$\kkk$ and we consider the unique extension to~$\kkk_g$ of the absolute value~$|\cdot|$ on~$\kkk$.)
If~$g$ is proximal, then its maximal eigenvalue belongs to~$\kkk$; we denote by $x_g^+\in\PP(V)$ the corresponding eigenline and by~$X_g^-$ the image in~$\PP(V)$ of the unique $g$-invariant complementary subspace of~$x_g^+$ in~$V$.
Note that $g$ acts on~$\PP(V)$ by contracting $\PP(V)\smallsetminus X_g^-$ towards~$x_g^+$.
For~$\varepsilon>0$, we will say that $g$ is $\varepsilon$-\textit{proximal} if it satisfies the two following additional conditions:
\begin{enumerate}
	\item $d(x_g^+,X_g^-)\geq 2\varepsilon$,
	\item for any $x\in\PP(V)$, if $d(x,X_g^-)\geq\varepsilon$, then $d(g\cdot x,x_g^+)\leq\varepsilon$.
\end{enumerate}
We will need the following lemma.

\begin{lem}\label{produits d'elements proximaux decales}
Let~$X_0^-$ be a projective hyperplane of~$\PP(V)$, let $x_0^+\in\PP(V)\smallsetminus\nolinebreak X_0^-$, and let $\varepsilon>0$ such that $d(x_0^+,X_0^-)\geq 2\varepsilon$.
Then there exists~$r_{\varepsilon}>0$ such that for any isometries $k_2,\ldots,k_n\in\End(V)$ with $d(k_i\cdot x_0^+,X_0^-)\geq 2\varepsilon$ and for any $\varepsilon$-proximal endomorphisms $z_1,\ldots,z_n\in\End(V)$ with $x_{z_i}^+=x_0^+$ and $X_{z_i}^-=X_0^-$, inducing a homothety of ratio~$\Vert z_i\Vert_{_V}$ on the line~$x_0^+$, we have
$$e^{-(n-1)\,r_{\varepsilon}}\cdot\prod_{i=1}^n \Vert z_i\Vert_{_V} \,\leq\, \Vert z_1k_2z_2\ldots k_n z_n\Vert_{_V} \,\leq\, \prod_{i=1}^n \Vert z_i\Vert_{_V}.$$
\end{lem}

\medskip

\begin{proof}
Since the operator norm~$\Vert\cdot\Vert_{_V}$ on~$\End(V)$ is submultiplicative and $k_i$ is an isometry of~$V$ for all~$i$,
$$\Vert z_1k_2z_2\ldots k_n z_n\Vert_{_V} \,\leq\, \prod_{i=1}^n \Vert z_i\Vert_{_V}.$$
Let us prove the left-hand inequality.
Let $v_0\in V\smallsetminus\{ 0\} $ satisfy $x_0^+=\kkk v_0$ and let~$V_0$ be the hyperplane of~$V$ such that $X_0^-=\PP(V_0)$.
Set
\begin{eqnarray*}
b_{\varepsilon} & = & \{ x\in\PP(V),\ d(x,x_0^+)\leq\varepsilon\} \quad\quad\ \,\\
\mathrm{and}\quad\quad B_{\varepsilon} & = & \{ x\in\PP(V),\ d(x,X_0^-)\geq\varepsilon\} .\quad\quad\ \,
\end{eqnarray*}
Note that the set of unitary vectors~$v\in V$ with $\kkk v\in B_{\varepsilon}$ is compact and that the map sending~$v\in V$ to $t\in\kkk$ such that $v\in tv_0+V_0$ is continuous, hence there exists~$r_{\varepsilon}>0$ such that
\begin{equation}\label{definition de r_epsilon}
v\in \big[e^{-\frac{r_{\varepsilon}}{2}},e^{\frac{r_{\varepsilon}}{2}}\big]\,\Vert v\Vert_{_V}\,v_0 + V_0
\end{equation}
for all~$v\in V\smallsetminus\{ 0\} $ with $\kkk v\in B_{\varepsilon}$.
Set
$$h_j = z_jk_{j+1}z_{j+1}\ldots k_nz_n\in\End(V)$$
for $1\leq j\leq n$.
We claim that $h_j\cdot B_{\varepsilon}\subset b_{\varepsilon}$ and
\begin{equation}\label{inegalite h_j}
\Vert h_j\cdot v_0\Vert_{_V} \,\geq\, e^{-(n-j)\,r_{\varepsilon}}\cdot\prod_{i=j}^n \Vert z_i\Vert_{_V}
\end{equation}
for all~$j$.
This follows from an easy descending induction on~$j$.
Indeed, for all~$i$ we have $k_i\cdot b_{\varepsilon}\subset B_{\varepsilon}$ since $k_i$ is an isometry of~$V$ and $d(k_i\cdot x_0^+,X_0^-)\geq 2\varepsilon$, and $z_i\cdot B_{\varepsilon}\subset b_{\varepsilon}$ since $z_i$ is $\varepsilon$-proximal with $x_{z_i}^+=x_0^+$ and $X_{z_i}^-=X_0^-$.
By~(\ref{definition de r_epsilon}), we have $k_{j+1}h_{j+1}\cdot v_0\in t_jv_0+V_0$ for some $t_j\in\R$ with
$$|t_j|\ \geq\ e^{-\frac{r_{\varepsilon}}{2}}\,\Vert k_{j+1}h_{j+1}\cdot v_0\Vert_{_V}\ =\ e^{-\frac{r_{\varepsilon}}{2}}\,\Vert h_{j+1}\cdot v_0\Vert_{_V}.$$
By the inductive assumption,
$$|t_j|\ \geq\ e^{-(n-j-\frac{1}{2})\,r_{\varepsilon}}\cdot\prod_{i=j+1}^n \Vert z_i\Vert_{_V}.$$
By hypothesis, $z_j$ preserves~$V_0$ and induces a homothety of ratio~$\Vert z_j\Vert_{_V}$ on the line~$x_0^+$, hence $h_j\cdot v_0=z_jk_{j+1}h_{j+1}\cdot v_0\in\Vert z_j\Vert_{_V}\,t_jv_0+V_0$, where
$$\Vert z_j\Vert_{_V}\,|t_j|\ \geq\ e^{-(n-j-\frac{1}{2})\,r_{\varepsilon}}\cdot\prod_{i=j}^n \Vert z_i\Vert_{_V}.$$
Inequality~(\ref{inegalite h_j}) follows, using~(\ref{definition de r_epsilon}) again.
\end{proof}

\subsection{Cartan projection along the fundamental weights}\label{Projection de Cartan and poids fondamentaux}

Lemma~\ref{produits d'elements proximaux decales} implies the following result.

\begin{prop}\label{mu d'un produit transverse selon alpha}
Let~$\kkk$ be a local field and $G$ the set of $\kkk$-points of a connected reductive algebraic $\kkk$-group.
Let $G=KA^+K$ or $G=KZ^+K$ be a Cartan decomposition and $\mu : G\rightarrow E^+$ the corresponding Cartan projection.
Fix~$\alpha\in\Delta$ and let $\mathcal{C}_{\alpha}$ be a compact subset of~$N_{\alpha}^-$.
Then there exist $r_{\alpha},R_{\alpha}>0$ such that for all $g_1,\ldots,g_n\in G$ with $\langle\alpha,\mu(g_i)\rangle\geq R_{\alpha}$ and $\ell_{g_i}k_{g_{i+1}}\in\mathcal{C}_{\alpha} P_{\alpha}$, we have
$$\bigg|\Big\langle\chi_{\alpha},\mu(g_1\ldots g_n) - \sum_{i=1}^n \mu(g_i)\Big\rangle\bigg| \leq n r_{\alpha}.$$
\end{prop}

We keep notation from Section~\ref{Preliminaires}.
In particular, for $g\in G$ we write $g = k_g z_g \ell_g$ with $k_g,\ell_g\in K$ and $z_g\in Z^+$, as in Subsection~\ref{Projection de Cartan}.
Given a simple root $\alpha\in\Delta$, we denote by~$\mathbf{N}_{\alpha}^-$ (resp.\ by~$\mathbf{P}_{\alpha}$) the unipotent (resp.\ parabolic) subgroup of~$\mathbf{G}$ introduced in Subsection~\ref{Sous-groupes paraboliques maximaux}, and by~$\chi_{\alpha}$ the highest weight of the representation $(V_{\alpha},\rho_{\alpha})$ introduced in Subsection~\ref{Representations de G}.

Precisely, Proposition~\ref{mu d'un produit transverse selon alpha} follows from Lemma~\ref{produits d'elements proximaux decales}, from (\ref{norme des representations and Cartan projection, cas reel}) and~(\ref{norme des representations and Cartan projection, cas ultrametrique}), and from the following lemma.

\begin{lem}\label{remarques V_alpha}
Let $x_{\alpha}^+\in\PP(V_{\alpha})$ be the highest weight line~$(V_{\alpha})_{\chi_{\alpha}}$, and let $X_{\alpha}^-$ be the image in~$\PP(V_{\alpha})$ of the sum of the weight spaces~$(V_{\alpha})_{\lambda}$ for $\lambda\in\Lambda_{\alpha}\smallsetminus\{ \chi_{\alpha}\} $.
\begin{enumerate}
	\item Given $\varepsilon>0$ with $d(x_{\alpha}^+,X_{\alpha}^-)\geq 2\varepsilon$, there exists~$R_{\alpha}>0$ such that for any $z\in Z^+$ with $\langle\alpha,\mu(z)\rangle\geq R_{\alpha}$, the element $\rho_{\alpha}(z)$ is $\varepsilon$-proximal in~$\PP(V_{\alpha})$ with $\big(x_{\rho_{\alpha}(z)}^+,X_{\rho_{\alpha}(z)}^-\big)=(x_{\alpha}^+,X_{\alpha}^-)$.
  \item We have $\rho_{\alpha}(N_{\alpha}^-)(x_{\alpha}^+)\cap X_{\alpha}^-=\emptyset$.
\end{enumerate}
\end{lem}

\begin{proof}
\begin{enumerate}
	\item It is sufficient to see that every restricted weight of~$(\rho_{\alpha},V_{\alpha})$ except~$\chi_{\alpha}$ belongs to $\chi_{\alpha}-\alpha-\N\Delta$.
  Consider the subgroup~$W_{\alpha}$ of~$W$ generated by the reflections $s_{\beta} : x\mapsto x-\langle x,\check{\beta}\rangle\,\beta$ for $\beta\in\Delta\smallsetminus\{ \alpha\} $.
  It acts transitively on the root subsystem of~$\Phi$ generated by~$\Delta\smallsetminus\{ \alpha\} $, and it fixes~$\chi_{\alpha}$ since $\chi_{\alpha}$ is a multiple of~$\omega_{\alpha}$ and $\langle\omega_{\alpha},\check{\beta}\rangle =0$ for all $\beta\in\Delta\smallsetminus\{ \alpha\} $.
  Therefore, for every weight $\lambda\in\chi_{\alpha}-\N(\Delta\smallsetminus\{ \alpha\} )$ there exists $w\in W_{\alpha}$ such that $w\cdot\lambda\in\chi_{\alpha}+\N\Delta$, which implies that $\lambda=\chi_{\alpha}$.
  \item For~$n\in N_{\alpha}^-$, the identity element~$1\in G$ belongs to the closure of the conjugacy class $\{ znz^{-1},\ z\in Z\} $, hence $x_{\alpha}^+$ belongs to the closure of the orbit $\rho_{\alpha}(Zn)(x_{\alpha}^+)$ in~$\PP(V_{\alpha})$.
  But $X_{\alpha}^-$ is closed in~$\PP(V_{\alpha})$, stable under~$Z$, and does not contain~$x_{\alpha}^+$.\qedhere
\end{enumerate}
\end{proof}

\smallskip

\begin{proof}[Proof of Proposition~\ref{mu d'un produit transverse selon alpha}]
The point $x_{\alpha}^+\in\PP(V_{\alpha})$ is fixed by~$P_{\alpha}$.
Moreover, $\rho_{\alpha}(N_{\alpha}^-)(x_{\alpha}^+)\cap X_{\alpha}^-=\emptyset$ by Lemma~\ref{remarques V_alpha}, hence there exists~$\varepsilon>0$ such that
$$d\big(\rho_{\alpha}(\mathcal{C}_{\alpha} P_{\alpha})(x_{\alpha}^+),X_{\alpha}^-\big)\geq 2\varepsilon.$$
Let $R_{\alpha}$ be given by Lemma~\ref{remarques V_alpha} and let $r_{\alpha}=r_{\varepsilon}/\log q$, where $r_{\varepsilon}>0$ is given by Lemma~\ref{produits d'elements proximaux decales} and $q=e$ if $\kkk$ is Archimedean, $q$ is the cardinal of the residue field of~$\mathcal{O}$ otherwise.
Let $g_1,\ldots,g_n\in G$ satisfy $\langle\alpha,\mu(g_i)\rangle\geq R_{\alpha}$ and $\ell_{g_i}k_{g_{i+1}}\in\mathcal{C}_{\alpha} P_{\alpha}$ for all~$i$.
By Lemma~\ref{remarques V_alpha}, $\rho_{\alpha}(z_{g_i})$ is $\varepsilon$-proximal in~$\PP(V_{\alpha})$ with $x_{\rho_{\alpha}(z_{g_i})}^+=x_{\alpha}^+$ and $X_{\rho_{\alpha}(z_{g_i})}^-=X_{\alpha}^-$.
Moreover, it induces a homothety of ratio~$\Vert\rho_{\alpha}(z_{g_i})\Vert_{\alpha}$ on the line~$x_{\alpha}^+$.
By Lemma~\ref{produits d'elements proximaux decales},
$$q^{-nr_{\alpha}}\cdot\prod_{i=1}^n \Vert\rho_{\alpha}(z_{g_i})\Vert_{\alpha}\ \leq\ \Vert\rho_{\alpha}(g_1\ldots g_n)\Vert_{\alpha}\ \leq\ \prod_{i=1}^n \Vert\rho_{\alpha}(z_{g_i})\Vert_{\alpha}.$$
Using~(\ref{norme des representations and Cartan projection, cas reel}) and~(\ref{norme des representations and Cartan projection, cas ultrametrique}), we get
$$\Big\langle\chi_{\alpha},\sum_{i=1}^n \mu(g_i)\Big\rangle - nr_{\alpha}\ \leq\ \langle\chi_{\alpha},\mu(g_1\ldots g_n)\rangle\ \leq\ \Big\langle\chi_{\alpha},\sum_{i=1}^n \mu(g_i)\Big\rangle.\qedhere$$
\end{proof}

\section{Transverse products}\label{Partie produit transverse}

In this section we explain how, under the assumptions of Theorem~\ref{varphi ne change pas beaucoup mu}, Proposition~\ref{mu d'un produit transverse selon alpha} applies to the elements $\gamma\in\nolinebreak\Gamma$ and their images~$\varphi(\gamma)$ under a small deformation $\varphi\in\Hom(\Gamma,G)$.
We use Guichard's idea~\cite{gui} of writing every element~$\gamma\in\Gamma$ as a ``transverse product'' $\gamma_0\ldots\gamma_n$ of elements of a fixed finite subset~$F$ of~$\Gamma$.

\subsection{Transversality in $L$}\label{Transversalite en rang un}

Let~$\kkk$ be a local field and $\mathbf{L}$ a connected reductive algebraic $\kkk$-group of $\kkk$-rank~$1$.
Fix a Cartan decomposition $L=K_LA_L^+K_L$ or $L=K_LZ_L^+K_L$, where $K_L$ is a maximal compact subgroup of~$L$, where $\mathbf{A_L}$ is a maximal $\kkk$-split $\kkk$-torus of~$\mathbf{L}$, and where $\mathbf{Z_L}$ is the centralizer of~$\mathbf{A_L}$ in~$\mathbf{L}$.
Let $\mu_L : L\rightarrow E_L^+$ denote the corresponding Cartan projection, where $E_L=Y(\mathbf{A_L})\otimes_{\Z}\R$.
Since $\mathbf{L}$ has $\kkk$-rank~$1$, the vector space~$E_L$ is a line, and any isomorphism from~$E_L$ to~$\R$ gives a Cartan projection $\mu_L^{\R} : L\rightarrow\R$.

If $\mathbf{L}$ has semisimple $\kkk$-rank~1, then~$\mu_L^{\R}$ takes only nonnegative or only nonpositive values.
We denote by~$\alpha_L$ the indivisible positive restricted root of~$\mathbf{A_L}$ in~$\mathbf{L}$, by $\mathbf{P_L}=\mathbf{P}_{\alpha_L}$ the proper parabolic subgroup of~$\mathbf{L}$ associated with~$\alpha_L$, and by~$\mathbf{N_L^-}=\mathbf{U}_{-\alpha_L}$ the unipotent subgroup associated with~$-\alpha_L$.

If $\mathbf{L}$ has semisimple $\kkk$-rank~$0$, then~$\mathbf{A_L}$ is central in~$\mathbf{L}$, hence $\mathbf{Z_L}=\mathbf{L}$.
In this case $\mu_L^{\R}$ is a group homomorphism from~$L$ to~$\R$, thus taking both positive and negative values.
We set $\mathbf{P_L}=\mathbf{Z_L}=\mathbf{L}$ and $\mathbf{N_L^-}=\{ 1\} $.

For the reader's convenience, we give a proof of the following result, which is due to Guichard in the real semisimple case (\cite{gui}, Lem.~7 \&~9).
We consider the more general situation of a reductive algebraic group over a local field.

\begin{prop}[Guichard]\label{produit d'elements transverses}
Let $\kkk$ be a local field, $L$ the set of $\kkk$-points of a connected reductive algebraic $\kkk$-group of $\kkk$-rank~$1$, and $\mu_L^{\R} : L\rightarrow\R$ a Cartan projection.
If $\kkk=\R$ or~$\C$, let $\Gamma$ be a convex cocompact subgroup of~$L$; if $\kkk$ is non-Archimedean, let $\Gamma$ be any finitely generated discrete subgroup of~$L$.
Then there exist~$D>0$ and a compact subset~$\mathcal{C}_L$ of~$N_L^-$ such that for $R\geq D$, any~$\gamma\in\Gamma$ may be written as $\gamma=\gamma_0\ldots\gamma_n$, where
\begin{enumerate}
  \item $|\mu_L^{\R}(\gamma_0)|\leq R+D$ and $R-D\leq |\mu_L^{\R}(\gamma_i)|\leq R+D$ for all $1\leq i\leq n$,
	\item $\mu_L^{\R}(\gamma_1),\ldots,\mu_L^{\R}(\gamma_n)$ are all~$\geq 0$ or all~$\leq 0$,
	\item $\ell_{\gamma_i} k_{\gamma_{i+1}} \in \mathcal{C}_L P_L$ for all $1\leq i\leq n-1$.
\end{enumerate}
\end{prop}

To prove Proposition~\ref{produit d'elements transverses} we use the following lemma, which translates the transversality condition~(3) in terms of~$\mu_L^{\R}$.

\begin{lem}\label{equivalence transversalite}
Under the assumptions of Proposition~\ref{produit d'elements transverses}, there exists $D_0\geq\nolinebreak 0$ with the following property: given any~$D\geq D_0$, there is a compact subset~$\mathcal{C}_L$ of~$N_L^-$ such that for $k\in K_L$, if
$$|\mu_L^{\R}(z_1kz_2)|\geq |\mu_L^{\R}(z_1)|+|\mu_L^{\R}(z_2)|-D$$
for some $z_1,z_2\in Z_L^+$ with $|\mu_L^{\R}(z_1)|,|\mu_L^{\R}(z_2)|\geq D$, then $k\in\mathcal{C}_L P_L$.
\end{lem}

\smallskip

Note that Proposition~\ref{mu d'un produit transverse selon alpha} implies some kind of converse to Lemma~\ref{equivalence transversalite}: for any compact subset~$\mathcal{C}_L$ of~$N_L^-$, there exists~$D\geq 0$ such that for all $k\in K_L\cap\mathcal{C}_LP_L$ and all $z_1,z_2\in Z^+$,
$$|\mu_L^{\R}(z_1kz_2)|\geq |\mu_L^{\R}(z_1)|+|\mu_L^{\R}(z_2)|-D.$$

\smallskip

\begin{proof}[Proof of Lemma~\ref{equivalence transversalite}]
We may assume that $\mathbf{L}$ has semisimple $\kkk$-rank~$1$.
Then $L/P_L$ is the disjoint union of $N_L^-\cdot P_L$ and $\{ w\cdot P_L\} $, where $w$ denotes the nontrivial element of the (restricted) Weyl group of~$L$.
It is therefore sufficient to prove the existence of a neighborhood~$\mathcal{U}$ of~$w\cdot P_L$ in~$L/P_L$ such that for all~$k\in K_L$ with $k\cdot P_L\in\mathcal{U}$ and all $z_1,z_2\in Z_L^+$ with $\mu_L^{\R}(z_1),\mu_L^{\R}(z_2)\geq D$, we have
$$|\mu_L^{\R}(z_1kz_2)| < |\mu_L^{\R}(z_1)| + |\mu_L^{\R}(z_2)| - D.$$
Let $X_L$ denote either the Riemannian symmetric space or the Bruhat-Tits tree of~$L$, depending on whether $\kkk$ is Archimedean or not.
The space $X_L$ is Gromov-hyperbolic and we may identify~$L/P_L$ with the boundary at infinity~$\partial X_L$ of~$X_L$, \textit{i.e.}, with the set of equivalence classes~$[\mathcal{R}]$ of geodesic half-lines $\mathcal{R} : [0,+\infty[\rightarrow X_L$ for the equivalence relation ``to stay at bounded distance''.
The point $P_L\in L/P_L$ (resp.\ $w\cdot P_L\in L/P_L$) corresponds to the equivalence class $[\mathcal{R}^+]$ (resp.\ $[\mathcal{R}^-]$) of the geodesic half-line $\mathcal{R}^+ : [0,+\infty[\rightarrow X$ (resp.\ $\mathcal{R}^- : [0,+\infty[\rightarrow X$) whose image is $Z_L^+\cdot x_0$ (resp.~$(w\cdot Z_L^+)\cdot x_0$).
Let $d$ be the distance on~$X_L$ and $x_0$ the point of~$X_L$ whose stabilizer is~$K_L$.
By~(\ref{prelim mu distance, reel}) and~(\ref{prelim mu distance, ultrametrique}), we may assume that $|\mu_L^{\R}(g)|=d(x_0,g\cdot x_0)$ for all~$g\in L$.
By the ``shadow lemma'' (see \cite{bou}, Lem.~1.6.2, for instance), there is a constant~$D_0>0$ such that the open sets
$$\mathcal{U}_t = \Big\{ [\mathcal{R}],\quad \mathcal{R}(0)=x_0\ \mathrm{and}\ d\big(\mathcal{R}(t),\mathcal{R}^-(t)\big)<D_0\Big\} ,$$
for $t\in [0,+\infty[ $, form a basis of neighborhoods of~$[\mathcal{R}^-]$ in~$\partial X_L$.
Fix $D\geq D_0$.
For all $k\in K_L$ and $z_1,z_2\in Z_L^+$ with $t_1:=\mu_L^{\R}(z_1)\geq D$ and $t_2:=\mu_L^{\R}(z_2)\geq D$, we have
\begin{eqnarray*}
|\mu_L^{\R}(z_1kz_2)| & = & d(x_0,z_1kz_2\cdot x_0)\\
& = & d(z_1^{-1}\cdot x_0,kz_2\cdot x_0)\\
& = & d\big(\mathcal{R}^-(t_1),k\cdot\mathcal{R}^+(t_2)\big)\\
& \leq & d\big(\mathcal{R}^-(t_1),\mathcal{R}^-(D)\big) + d\big(\mathcal{R}^-(D),k\cdot\mathcal{R}^+(D)\big)\\
& & \quad\ +\, d\big(k\cdot\mathcal{R}^+(D),k\cdot\mathcal{R}^+(t_2)\big)\\
& = & t_1 - D + d\big(\mathcal{R}^-(D),k\cdot\mathcal{R}^+(D)\big) + t_2 - D\\
& = & |\mu_L^{\R}(z_1)| + |\mu_L^{\R}(z_2)| - 2D + d\big(\mathcal{R}^-(D),k\cdot\mathcal{R}^+(D)\big).
\end{eqnarray*}
Therefore, if $[k\cdot\mathcal{R}^+]\in\mathcal{U}_D$ then $|\mu_L^{\R}(z_1kz_2)|<|\mu_L^{\R}(z_1)|+|\mu_L^{\R}(z_2)|-D$.
This completes the proof of Lemma~\ref{equivalence transversalite}.
\end{proof}

\smallskip

\begin{proof}[Proof of Proposition~\ref{produit d'elements transverses}]
As in the proof of Lemma~\ref{equivalence transversalite}, let~$X_L$ denote either the Riemannian symmetric space or the Bruhat-Tits tree of~$L$, depending on whether $\kkk$ is Archimedean or not.
Let $d$ be the distance on~$X_L$ and $x_0$ the point of~$X_L$ whose stabilizer is~$K_L$.
By~(\ref{prelim mu distance, reel}) and~(\ref{prelim mu distance, ultrametrique}), we may assume that $|\mu_L^{\R}(g)|=d(g\cdot x_0,x_0)$ for all~$g\in L$.
Let~$X'_L$ denote the convex hull of the limit set of~$\Gamma$ in~$X_L$.
It is a closed subset of~$X_L$ on which $\Gamma$ acts cocompactly: indeed, if $\kkk=\R$ or~$\C$ this is the convex cocompacity assumption; if $\kkk$ is non-Archimedean it follows from~\cite{bas}, Prop.~7.9.
Fix a compact fundamental domain~$\D$ of~$X'_L$ for the action of~$\Gamma$, and fix~$x'_0\in\Int(\D)$.
Let~$\mathrm{d}_{\D}$ be the diameter of~$\D$ and $D_0$ the constant given by Lemma~\ref{equivalence transversalite}.
Let
$$D = \max\big(D_0,6\,\mathrm{d}_{\D}+6\,d(x_0,x'_0)\big) > 0$$
and let $\mathcal{C}_L$ be the corresponding compact subset of~$N_L^-$ given by Lemma~\ref{equivalence transversalite}.
We claim that $D$ and $\mathcal{C}_L$ satisfy the conclusions of Proposition~\ref{produit d'elements transverses}.
Indeed, let~$R\geq D$.
Fix~$\gamma\in\Gamma$ and let~$I$ be the geodesic segment of~$X'_L$ with endpoints $x'_0$ and $\gamma^{-1}\cdot x'_0$.
Let $n\in\N$ such that
$$nR\, \leq\, d(x'_0,\gamma^{-1}\cdot x'_0)\, <\, (n+1)R.$$
For all $1\leq i\leq n$, let $x'_i\in I$ satisfy $d(x'_i,x'_0)=iR$.
We have $x'_i\in\lambda_i\cdot\D$ for some $\lambda_i\in\Gamma$.
Let $\gamma_0=\gamma\lambda_n\in\Gamma$ and $\gamma_i=\lambda_{n-i+1}^{-1}\lambda_{n-i}\in\Gamma$ for $i\geq 1$ (where $\lambda_0=1$), so that $\gamma=\gamma_0\ldots\gamma_n$.
For all $1\leq i\leq n$,
\begin{eqnarray*}
\big||\mu_L^{\R}(\gamma_i)| - d(x'_{n-i},x'_{n-i+1})\big| & = & \big|d(\lambda_{n-i}\cdot x_0,\lambda_{n-i+1}\cdot x_0) - d(x'_{n-i},x'_{n-i+1})\big|\\
& \leq & d(\lambda_{n-i}\cdot x_0,\lambda_{n-i}\cdot x'_0) + d(\lambda_{n-i}\cdot x'_0,x'_{n-i})\\
& & +\, d(x'_{n-i+1},\lambda_{n-i+1}\cdot x'_0) + d(\lambda_{n-i+1}\cdot x'_0,\lambda_{n-i+1}\cdot x_0)\\
& \leq & 2\,\mathrm{d}_{\D} + 2\,d(x_0,x'_0).
\end{eqnarray*}
Since $d(x'_{n-i},x'_{n-i+1})=R$, we have $\big||\mu_L^{\R}(\gamma_i)|-R\big| \leq 2\,\mathrm{d}_{\D} + 2\,d(x_0,x'_0)$.
Similarly,
$$\big||\mu_L^{\R}(\gamma_0)| - d(x'_n,\gamma^{-1}\cdot x'_0)\big| \leq 2\,\mathrm{d}_{\D} + 2\,d(x_0,x'_0),$$
hence $|\mu_L^{\R}(\gamma_0)|\leq R+2\,\mathrm{d}_{\D}+2\,d(x_0,x'_0)$.
For $1\leq i\leq n-1$, the same reasoning shows that
\begin{eqnarray}\label{comparaison entre mu(gamma_i gamma_i+1) et mu(gamma_i) + mu(gamma_i+1)}
|\mu_L^{\R}(\gamma_i\gamma_{i+1})| & \geq & d(x_{n-i-1},x_{n-i+1}) - 2\,\mathrm{d}_{\D} - 2\,d(x_0,x'_0)\nonumber\\
& = & 2R - 2\,\mathrm{d}_{\D} - 2\,d(x_0,x'_0)\nonumber\\
& \geq & |\mu_L^{\R}(\gamma_i)| + |\mu_L^{\R}(\gamma_{i+1})| - 6\,\mathrm{d}_{\D} - 6\,d(x_0,x'_0)\nonumber\\
& \geq & |\mu_L^{\R}(\gamma_i)| + |\mu_L^{\R}(\gamma_{i+1})| - D.
\end{eqnarray}
By Lemma~\ref{equivalence transversalite}, we have $\ell_{\gamma_i}k_{\gamma_{i+1}}\in \mathcal{C}_L P_L$ for all $1\leq i\leq n-1$.

We claim that $\mu_L^{\R}(\gamma_1),\ldots,\mu_L^{\R}(\gamma_n)\in\R$ all have the same sign.
Indeed, we may assume that $\mathbf{L}$ has semisimple $\kkk$-rank~$0$, in which case $\mu_L^{\R} : L\rightarrow\R$ is a group homomorphism.
If $\mu_L^{\R}(\gamma_i)$ and~$\mu_L^{\R}(\gamma_{i+1})$ had different signs for some $1\leq i\leq\nolinebreak n-1$, then (\ref{comparaison entre mu(gamma_i gamma_i+1) et mu(gamma_i) + mu(gamma_i+1)}) would imply that $$\min\big(|\mu_L^{\R}(\gamma_i)|,|\mu_L^{\R}(\gamma_{i+1})|\big) \leq \frac{D}{2},$$
which would contradict the fact that
\begin{eqnarray*}
|\mu_L^{\R}(\gamma_i)|,|\mu_L^{\R}(\gamma_{i+1})| & \geq & R-2\,\mathrm{d}_{\D}-2\,d(x_0,x'_0)\\
& \geq & D-2\,\mathrm{d}_{\D}-2\,d(x_0,x'_0)\ >\ \frac{D}{2}.\qedhere
\end{eqnarray*}
\end{proof}

\subsection{Transversality in $G$}\label{Transversalite en rang superieur}

Let $\kkk$ be a local field, $\mathbf{G}$ a connected reductive algebraic $\kkk$-group, and $\mathbf{L}$ a closed connected reductive subgroup of~$\mathbf{G}$ of $\kkk$-rank~$1$.
Fix a Cartan decomposition $G=KA^+K$ or $G=KZ^+K$.

\begin{rema}\label{decompositions de Cartan compatibles}
After conjugating~$\mathbf{L}$ by some element of~$G$, we may assume that $L$ admits a Cartan decomposition $L=K_LA_L^+K_L$ or $L=K_LZ_L^+K_L$ with $K_L\subset K$, with $\mathbf{A_L}\subset\mathbf{A}$, and with $A_L^+\cap A^+$ noncompact.
\end{rema}

\noindent
Indeed, $\mathbf{A_L}$ is contained in some maximal $\kkk$-split $\kkk$-torus of~$\mathbf{G}$, and these tori are all conjugate over~$\kkk$ (\cite{bot}, Th.~4.21).
Thus, after conjugating~$\mathbf{L}$ by some element of~$G$, we may assume that $\mathbf{A_L}\subset\mathbf{A}$.
We now use a result proved by Mostow~\cite{mos55} and Karpelevich~\cite{kar} in the Archimedean case, and by Landvogt~\cite{lan} in the non-Archimedean case: after conjugating~$\mathbf{L}$ again by some element of~$G$, we may assume that $K_L\subset K$.
Finally, after conjugating~$\mathbf{L}$ by some element of the Weyl group~$W$, we may assume that $A_L^+\cap A^+$ is noncompact.

\medskip

Assume that the conditions of Remark~\ref{decompositions de Cartan compatibles} are satisfied.
The following lemma provides a link between Propositions~\ref{mu d'un produit transverse selon alpha} and~\ref{produit d'elements transverses}.
We use the notation of Subsection~\ref{Sous-groupes paraboliques maximaux}.

\begin{lem}\label{P_L and P_alpha}
If the restriction of $\alpha\in\Delta$ to~$\mathbf{A_L}$ is nontrivial, then $P_L\subset P_{\alpha}$ and $N_L^- \subset N_{\alpha}^- P_{\alpha}$.
\end{lem}

\begin{proof}
Fix~$\alpha\in\Delta$ whose restriction to~$\mathbf{A_L}$ is nontrivial, and let $a\in A_L^+\cap A^+$ such that $|\alpha(a)|>1$.
Note that $\g=\n_{\alpha}^-\oplus\p_{\alpha}$ and $\p_{\alpha}=\p_{\emptyset}\oplus\n_{\alpha^c}^-$, where
\begin{eqnarray*}
\n_{\alpha}^- & = & \bigoplus_{\beta\in\Phi^+} \uu_{-(\alpha+\beta)},\\
\p_{\emptyset} & = & \z \oplus \bigoplus_{\beta\in\Phi^+} \uu_{\beta},\\
\mathrm{and}\quad\quad \n_{\alpha^c}^- & = & \bigoplus_{\beta\in\N(\Delta\smallsetminus\{ \alpha\} )} \uu_{-\beta}
\end{eqnarray*}
are all direct sums of eigenspaces of~$\Ad(a)$, with eigenvalues of absolute value~$<1$ on~$\n_{\alpha}^-$ and~$\geq 1$ on~$\p_{\emptyset}$.
Since $\p_L$ is a sum of eigenspaces of~$\Ad(a)$ for eigenvalues of absolute value~$\geq 1$, we have $\p_L\subset\p_{\alpha}$.
Given that $\mathbf{P_L}$ and~$\mathbf{P}_{\alpha}$ are connected, this implies that~$P_L\subset P_{\alpha}$.
Since $\n_L^-$ is a sum of eigenspaces of~$\Ad(a)$ for eigenvalues of absolute value~$<1$, we have $\n_L^-\subset\n_{\alpha}^-\oplus\n_{\alpha^c}^-$.
Note that $[\n_{\alpha}^-,\n_{\alpha^c}^-] \subset \n_{\alpha}^-$, hence $N_{\alpha}^-$ is normalized by the group~$N_{\alpha^c}^-$ generated by the groups~$U_{-\beta}$ for $\beta\in\N(\Delta\smallsetminus\{ \alpha\} )$.
This implies that
$$N_L^- \,\subset\, N_{\alpha}^- N_{\alpha^c}^- \,\subset\, N_{\alpha}^- P_{\alpha}.\qedhere$$
\end{proof}

\section{Cartan projection and deformation}\label{Demonstration des theoremes}

In this section we prove Theorem~\ref{varphi ne change pas beaucoup mu} using Propositions~\ref{mu d'un produit transverse selon alpha} and~\ref{produit d'elements transverses}.
By the triangular inequality, it is sufficient to prove the following proposition.

\begin{prop}\label{mu d'un produit et somme des mu}
Let $\kkk$ be a local field, $G$ the set of $\kkk$-points of a connected reductive algebraic $\kkk$-group~$\mathbf{G}$, and $L$ the set of $\kkk$-points of a closed reductive subgroup~$\mathbf{L}$ of~$\mathbf{G}$ of $\kkk$-rank~$1$.
Fix a Cartan projection $\mu : G\rightarrow E^+$ and a norm~$\Vert\cdot\Vert$ on~$E$.
If $\kkk=\R$ or~$\C$, let $\Gamma$ be a convex cocompact subgroup of~$L$; if $\kkk$ is non-Archimedean, let $\Gamma$ be any finitely generated discrete subgroup of~$L$.
Then for any~$\varepsilon>0$, there exist a finite subset~$F_{\varepsilon}$ of~$\Gamma$, a neighborhood $\mathcal{U}_{\varepsilon}\subset\Hom(\Gamma,G)$ of the natural inclusion, and a constant $C_{\varepsilon}\geq 0$ such that any~$\gamma\in\Gamma$ may be written as $\gamma=\gamma_0\ldots\gamma_n$ for some $\gamma_0,\ldots,\gamma_n\in F_{\varepsilon}$ with
\begin{enumerate}
	\item $n \leq \varepsilon\,\Vert\mu(\gamma)\Vert + C_{\varepsilon}$,
	\item $\Vert\mu(\varphi(\gamma_i))-\mu(\gamma_i)\Vert\leq 1$ for all $\varphi\in\mathcal{U}_{\varepsilon}$ and $0\leq i\leq n$,
  \item for all $\varphi\in\mathcal{U}_{\varepsilon}$,
  $$\Big\Vert\mu(\varphi(\gamma)) - \sum_{i=0}^n \mu(\varphi(\gamma_i))\Big\Vert \leq \varepsilon \Vert\mu(\gamma)\Vert + C_{\varepsilon}.$$
\end{enumerate}
\end{prop}

The proof of Proposition~\ref{mu d'un produit et somme des mu} will be given in Subsection~\ref{Demonstration de la proposition sur mu}.

\subsection{Norms on~$E$ and its subspaces}

Under the assumptions of Proposition~\ref{mu d'un produit et somme des mu}, let $G=KA^+K$ or $G=KZ^+K$ be the Cartan decomposition corresponding to~$\mu$.
By~(\ref{inegalite fine pour mu}), in order to prove Proposition~\ref{mu d'un produit et somme des mu}, we may assume that~$\mathbf{L}$ is connected and replace it by any conjugate by~$G$.
By Remark~\ref{decompositions de Cartan compatibles}, we may assume that $L$ admits a Cartan decomposition $L=K_LA_L^+K_L$ or $L=K_LZ_L^+K_L$ with $K_L\subset K$, with $\mathbf{A_L}\subset\mathbf{A}$, and with $A_L^+\cap A^+$ noncompact.
Let $\mu_L : L\rightarrow E_L^+$ be the corresponding Cartan projection.
We naturally see~$E_L$ as a line in~$E$.
If $\mathbf{L}$ has semisimple $\kkk$-rank~$1$, then $E_L^+$ is a half-line in~$E^+$; if $\mathbf{L}$ has semisimple $\kkk$-rank~$0$, then $E_L^+=E_L$ is a line in~$E$, intersecting~$E^+$ in a half-line, and
$$\mu(g) = E^+\cap W\cdot\mu_L(g)$$
for all~$g\in L$.
Since all norms on~$E$ are equivalent, we may assume that $\Vert\cdot\Vert$ is the $W$-invariant Euclidean norm introduced in Section~\ref{Preliminaires}.
By composing~$\mu_L$ with some isomorphism from~$E_L$ to~$\R$, we get a Cartan projection $\mu_L^{\R} : L\rightarrow\R$ with
$$|\mu_L^{\R}(g)|=\Vert\mu(g)\Vert$$
for all~$g\in L$.
For every~$\alpha\in\Delta$ there are constants $t_{\alpha}^+,t_{\alpha}^-\geq 0$ such that
\begin{equation}\label{definition t_alpha}
  \langle\alpha,\mu(g)\rangle =
     \begin{cases}
        t_{\alpha}^+\,|\mu_L^{\R}(g)| & \text{if $\mu_L^{\R}(g)\geq 0$,}\\
        t_{\alpha}^-\,|\mu_L^{\R}(g)| & \text{if $\mu_L^{\R}(g)\leq 0$.}
     \end{cases}
\end{equation}
Let $\Delta_L = \{ \alpha\in\Delta,\ t_{\alpha}^{\pm}>0\} $ denote the set of simple roots of~$\mathbf{A}$ in~$\mathbf{G}$ whose restriction to~$\mathbf{A_L}$ is nontrivial.
Let $E_{\Delta_L}$ denote the subspace of~$E$ spanned by the coroots~$\check{\alpha}$ for $\alpha\in\Delta_L$, and let $\pr_{E_{\Delta_L}} : E\rightarrow E_{\Delta_L}$ denote the orthogonal projection on~$E_{\Delta_L}$.
Then
$$|v|_{E_{\Delta_L}} = \Vert\pr_{E_{\Delta_L}}(v)\Vert$$
defines a seminorm~$|\cdot|_{E_{\Delta_L}}$ on~$E$.
For $\alpha\in\Delta$, let~$\chi_{\alpha}$ denote the highest weight of the representation~$(\rho_{\alpha},V_{\alpha})$ of~$\mathbf{G}$ introduced in Subsection~\ref{Representations de G}.
Recall that $\langle\chi_{\alpha},\check{\alpha}\rangle\neq 0$ and $\langle\chi_{\alpha},\check{\beta}\rangle= 0$ for all $\beta\in\nolinebreak\Delta\smallsetminus\nolinebreak\{ \alpha\} $, hence $\{ \langle\chi_{\alpha},\cdot\rangle,\ \alpha\in\Delta_L\} $ is a basis of the dual of~$E_{\Delta_L}$.
Thus the function
$$v \longmapsto \sum_{\alpha\in\Delta_L} |\langle\chi_{\alpha},v\rangle|$$
is a norm on~$E_{\Delta_L}$.
Since all norms on~$E_{\Delta_L}$ are equivalent, there exists~$c\geq 1$ such that
\begin{equation}\label{normes equivalentes}
c^{-1}\cdot\sum_{\alpha\in\Delta_L} |\langle\chi_{\alpha},v\rangle|\ \leq\ |v|_{E_{\Delta_L}}\ \leq\ c\cdot\sum_{\alpha\in\Delta_L} |\langle\chi_{\alpha},v\rangle|
\end{equation}
for all~$v\in E$.

\subsection{Norm of the projection on~$E_{\Delta_L}$}

The main step in the proof of Proposition~\ref{mu d'un produit et somme des mu} consists of the following proposition, which gives an upper bound for the seminorm~$|\cdot|_{E_{\Delta_L}}$.

\begin{prop}\label{lemme majoration E_Delta_L}
Under the assumptions of Proposition~\ref{mu d'un produit et somme des mu}, for any~$\delta>0$ there exist a finite subset~$F'_{\delta}$ of~$\Gamma$, a neighborhood $\mathcal{U}'_{\delta}\subset\Hom(\Gamma,G)$ of the natural inclusion, and a constant~$C'_{\delta}\geq 0$ such that any~$\gamma\in\Gamma$ may be written as $\gamma=\gamma_0\ldots\gamma_n$ for some $\gamma_0,\ldots,\gamma_n\in F'_{\delta}$ with
\begin{enumerate}
  \item $n\leq\delta\sum_{i=0}^n \Vert\mu(\gamma_i)\Vert$,
	\item $\sum_{i=1}^n \Vert\mu(\gamma_i)\Vert = \Vert\sum_{i=1}^n \mu(\gamma_i)\Vert$,
	\item $\Vert\mu(\varphi(\gamma_i))-\mu(\gamma_i)\Vert\leq 1$ for all~$\varphi\in\mathcal{U}'_{\delta}$ and $0\leq i\leq n$,
	\item for all~$\varphi\in\mathcal{U}'_{\delta}$,
$$\Big|\mu(\varphi(\gamma)) - \sum_{i=0}^n \mu(\varphi(\gamma_i))\Big|_{E_{\Delta_L}} \leq \delta\,\Big(\sum_{i=0}^n \Vert\mu(\gamma_i)\Vert\Big) + C'_{\delta}.$$
\end{enumerate}
\end{prop}

To prove Proposition~\ref{lemme majoration E_Delta_L}, we use Propositions~\ref{mu d'un produit transverse selon alpha} and~\ref{produit d'elements transverses}, together with Lemma~\ref{P_L and P_alpha}.

\begin{proof}
Let $D>0$ be the constant and $\mathcal{C}_L$ the compact subset of~$N_L^-$ given by Proposition~\ref{produit d'elements transverses}.
By Lemma~\ref{P_L and P_alpha}, for any $\alpha\in\Delta_L$, the set~$\mathcal{C}_L$ is contained in $\Int(\mathcal{C}_{\alpha})P_{\alpha}$ for some compact subset~$\mathcal{C}_{\alpha}$ of~$N_{\alpha}^-$, where $\Int(\mathcal{C}_{\alpha})$ denotes the interior of~$\mathcal{C}_{\alpha}$.
Let $r_{\alpha},R_{\alpha}>0$ be the corresponding constants given by Proposition~\ref{mu d'un produit transverse selon alpha}.
Fix $\delta>0$ and choose $R>D$ large enough so that $\frac{1}{R-D} \leq \delta$ and $\min(t_{\alpha}^+,t_{\alpha}^-)(R-D) - 1\ \geq\ R_{\alpha}$ for all~$\alpha\in\Delta_L$, where $t_{\alpha}^+$ and~$t_{\alpha}^-$ are defined by~(\ref{definition t_alpha}).
Let $F'_{\delta}$ be the set of elements~$\gamma\in\Gamma$ such that $|\mu_L^{\R}(\gamma)|\leq R+D$, and $F''_{\delta}$ the subset of elements $\gamma\in F'_{\delta}$ such that $|\mu_L^{\R}(\gamma)|\geq R-D$.
Note that $F'_{\delta}$ et~$F''_{\delta}$ are finite since $\mu_L^{\R}$ is a proper map and $\Gamma$ is discrete in~$L$.
Let $\mathcal{U}'_{\delta}\subset\Hom(\Gamma,G)$ be the neighborhood of the natural inclusion whose elements~$\varphi$ satisfy the following two conditions:
\begin{itemize}
	\item $\Vert\mu(\varphi(\gamma))-\mu(\gamma)\Vert\leq 1$ and $|\langle\alpha,\mu(\varphi(\gamma))-\mu(\gamma)\rangle|\leq 1$ for all~$\gamma\in F'_{\delta}$ and all~$\alpha\in\Delta_L$,
	\item $\ell_{\varphi(\gamma)}k_{\varphi(\gamma')}\in \mathcal{C}_{\alpha} P_{\alpha}$ for all~$\gamma,\gamma'\in F''_{\delta}$ with $\ell_{\gamma}k_{\gamma'}\in \mathcal{C}_L P_L$ and all~$\alpha\in\nolinebreak\Delta_L$,
\end{itemize}
where for $g\in G$ we write $g = k_g z_g \ell_g$ with $k_g,\ell_g\in K$ and $z_g\in Z^+$.
We claim that $F'_{\delta}$ and $\mathcal{U}'_{\delta}$ satisfy the conclusions of Proposition~\ref{lemme majoration E_Delta_L} for some constant~$C'_{\delta}$.
Indeed, let~$\gamma\in\Gamma$.
By Proposition~\ref{produit d'elements transverses}, we may write $\gamma=\gamma_0\ldots\gamma_n$ for some elements $\gamma_0\in F'_{\delta}$ and $\gamma_1,\ldots,\gamma_n\in F''_{\delta}$ such that
\begin{itemize}
	\item $\ell_{\gamma_i} k_{\gamma_{i+1}} \in \mathcal{C}_L P_L$ for all $1\leq i\leq n-1$,
	\item $\mu_L^{\R}(\gamma_1),\ldots,\mu_L^{\R}(\gamma_n)$ are all~$\geq 0$ or all~$\leq 0$.
\end{itemize}
This last condition implies that $\mu_L^{\R}(\gamma_1),\ldots,\mu_L^{\R}(\gamma_n)$ all belong to the same half-line in~$E^+$, hence
$$\sum_{i=1}^n \Vert\mu(\gamma_i)\Vert = \Big\Vert\sum_{i=1}^n \mu(\gamma_i)\Big\Vert.$$
Moreover, since $\gamma_1,\ldots,\gamma_n\in F''_{\delta}$, we have
$$n \,\leq\, \frac{1}{R-D}\cdot\sum_{i=0}^n |\mu_L^{\R}(\gamma_i)| \,\leq\, \delta\,\sum_{i=0}^n \Vert\mu(\gamma_i)\Vert.$$
Let~$\varphi\in\mathcal{U}'_{\delta}$.
According to~(\ref{normes equivalentes}), in order to prove Condition~(4) it is sufficient to bound
$$\bigg|\Big\langle\chi_{\alpha},\mu(\varphi(\gamma)) - \sum_{i=0}^n \mu(\varphi(\gamma_i))\Big\rangle\bigg|$$
for all~$\alpha\in\Delta_L$.
For~$\alpha\in\Delta_L$ and $1\leq i\leq n$ we have
\begin{eqnarray*}
\langle\alpha,\mu(\varphi(\gamma_i))\rangle & \geq & \langle\alpha,\mu(\gamma_i)\rangle - 1\\
& \geq & \min(t_{\alpha}^+,t_{\alpha}^-)\,|\mu_L^{\R}(\gamma_i)| - 1\\
& \geq & \min(t_{\alpha}^+,t_{\alpha}^-)(R-D) - 1\ \geq\ R_{\alpha}
\end{eqnarray*}
and $\ell_{\varphi(\gamma_i)} k_{\varphi(\gamma_{i+1})} \in \mathcal{C}_{\alpha} P_{\alpha}$.
Proposition~\ref{mu d'un produit transverse selon alpha} thus implies that
\begin{eqnarray*}
\bigg|\Big\langle\chi_{\alpha},\mu(\varphi(\gamma_1\ldots\gamma_n)) - \sum_{i=1}^n \mu(\varphi(\gamma_i))\Big\rangle\bigg| & \leq & nr_{\alpha}\\
& \leq & \frac{r_{\alpha}}{R-D}\,\sum_{i=0}^n \Vert\mu(\gamma_i)\Vert.
\end{eqnarray*}
On the other hand, by (\ref{normes equivalentes}) and~(\ref{inegalite fine pour mu}),
\begin{eqnarray*}
\big|\big\langle\chi_{\alpha},\mu(\varphi(\gamma))-\mu(\varphi(\gamma_1\ldots\gamma_n))\big\rangle\big|
& \leq & c\,\Vert\mu(\varphi(\gamma_0))\Vert\\
& \leq & c \cdot \max_{f\in F'_{\delta}} \big(\Vert\mu(f)\Vert + 1\big).
\end{eqnarray*}
By the triangular inequality, we finally get
$$\bigg|\Big\langle\chi_{\alpha},\mu(\varphi(\gamma)) - \sum_{i=0}^n \mu(\varphi(\gamma_i))\Big\rangle\bigg| \leq \frac{r_{\alpha}}{R-D}\,\Big(\sum_{i=0}^n \Vert\mu(\gamma_i)\Vert\Big) + C'_{\delta}$$
with
$$C'_{\delta} = c \cdot \max_{f\in F'_{\delta}} \big(\Vert\mu(f)\Vert + |\langle\chi_{\alpha},\mu(f)\rangle| + 2\big).$$
By~(\ref{normes equivalentes}), this implies Condition~(4) whenever
$$\sum_{\alpha\in\Delta_L} \frac{c\,r_{\alpha}}{R-D}\ \leq\ \delta,$$
which holds for $R$ large enough.
\end{proof}

\subsection{Proof of Proposition~\ref{mu d'un produit et somme des mu}}\label{Demonstration de la proposition sur mu}

Proposition~\ref{mu d'un produit et somme des mu} follows from Proposition~\ref{lemme majoration E_Delta_L} and from the following general observation.

\begin{lem}\label{remarque formelle}
Let $(E,\Vert\cdot\Vert)$ be a Euclidean space, with corresponding scalar product~$\langle\cdot,\cdot\rangle$, and $E_1$ a subspace of~$E$.
For any~$x\in E$, let $|x|_1=\Vert\pr_1(x)\Vert$, where $\pr_1 : E\rightarrow E_1$ denotes the orthogonal projection on~$E_1$.
For any~$\delta>0$, fix $C''_{\delta}\geq 0$ and let~$\mathcal{I}_{\delta}$ be the set of pairs $(x,x')\in E^2$ such that
\begin{enumerate}
  \item $x\in E_1$,
	\item $|x'-x|_1\leq 2\delta\,\Vert x\Vert+C''_{\delta}$,
	\item $\Vert x'\Vert\leq (1+\delta)\,\Vert x\Vert+C''_{\delta}$.
\end{enumerate}
Then
$$\sup_{(x,x')\in\mathcal{I}_{\delta},\ x\neq 0} \frac{\Vert x'-x\Vert - 4C''_{\delta}/\delta}{\Vert x\Vert} \longrightarrow_{\delta\rightarrow 0} 0.$$
\end{lem}

\begin{proof}
Let $\delta\in ]0,1]$ and $(x,x')\in\mathcal{I}_{\delta}$.
If $\Vert x\Vert\leq C''_{\delta}/\delta$, then
$$\Vert x'-x\Vert \leq \Vert x'\Vert + \Vert x\Vert \leq (2+\delta)\,\Vert x\Vert + C''_{\delta} \leq \frac{4C''_{\delta}}{\delta}.$$
If $\Vert x\Vert\geq C''_{\delta}/\delta$, then
$$\Vert x'\Vert \leq (1+2\delta)\,\Vert x\Vert$$
and
$$\langle x,x'\rangle = \langle x,\pr_1(x')\rangle \,\geq\, \Vert x\Vert^2 - \Vert x\Vert\,|x-x'|_1 \,\geq\, (1-3\delta)\,\Vert x\Vert^2.$$
To conclude, note that for $y\in E$ with $\Vert y\Vert=1$, the diameter of the set
$$\big\{ y'\in E,\quad 1-3\delta \,\leq\, \langle y,y'\rangle \,\leq\, \Vert y'\Vert \,\leq\, 1+2\delta\big\} $$
uniformly tends to~$0$ with~$\delta$.
\end{proof}

\medskip

\begin{proof}[Proof of Proposition~\ref{mu d'un produit et somme des mu}]
Fix~$\varepsilon\in ]0,1]$.
For $\delta\in ]0,1]$, let $F'_{\delta}$, $\mathcal{U}'_{\delta}$, and~$C'_{\delta}$ be given by Proposition~\ref{lemme majoration E_Delta_L}.
Let $\gamma\in\Gamma$.
By Proposition~\ref{lemme majoration E_Delta_L}, we may write $\gamma=\gamma_0\ldots\gamma_n$ for some $\gamma_0,\ldots,\gamma_n\in F'_{\delta}$ such that
\begin{enumerate}
  \item $n\leq\delta\sum_{i=0}^n \Vert\mu(\gamma_i)\Vert$,
	\item $\sum_{i=1}^n \Vert\mu(\gamma_i)\Vert = \Vert\sum_{i=1}^n \mu(\gamma_i)\Vert$,
	\item $\Vert\mu(\varphi(\gamma_i))-\mu(\gamma_i)\Vert\leq 1$ for all~$\varphi\in\mathcal{U}'_{\delta}$ and $0\leq i\leq n$,
	\item for all~$\varphi\in\mathcal{U}'_{\delta}$,
	$$\Big|\mu(\varphi(\gamma)) - \sum_{i=0}^n \mu(\varphi(\gamma_i))\Big|_{E_{\Delta_L}} \leq \delta\,\Big(\sum_{i=0}^n \Vert\mu(\gamma_i)\Vert\Big) + C'_{\delta}.$$
\end{enumerate}
Conditions (1), (3), and~(4), together with~(\ref{inegalite triangulaire pour mu}) and the triangular inequality, imply that for all~$\varphi\in\mathcal{U}'_{\delta}$,
\begin{eqnarray*}
\Big|\mu(\varphi(\gamma)) - \sum_{i=0}^n \mu(\gamma_i)\Big|_{E_{\Delta_L}} & \leq & \Big|\mu(\varphi(\gamma)) - \sum_{i=0}^n \mu(\varphi(\gamma_i))\Big|_{E_{\Delta_L}} + \sum_{i=0}^n \Vert\mu(\varphi(\gamma_i))-\mu(\gamma_i)\Vert\\
& \leq & 2\delta\,\Big(\sum_{i=0}^n \Vert\mu(\gamma_i)\Vert\Big) + C'_{\delta} + 1
\end{eqnarray*}
and
$$\Vert\mu(\varphi(\gamma))\Vert \,\leq\, \sum_{i=0}^n \Vert\mu(\varphi(\gamma_i))\Vert \,\leq\, (1+\delta)\,\Big(\sum_{i=0}^n \Vert\mu(\gamma_i)\Vert\Big) + 1.$$
Moreover, Condition~(2) implies that
\begin{equation}\label{norme de la somme des mu(gamma_i)}
\sum_{i=0}^n \Vert\mu(\gamma_i)\Vert \,\leq\, \Big\Vert\sum_{i=0}^n \mu(\gamma_i)\Big\Vert + 2\,\max_{g\in F'_{\delta}} \Vert\mu(g)\Vert.
\end{equation}
Therefore, for $\varphi\in\mathcal{U}'_{\delta}$, Lemma~\ref{remarque formelle} applies to
\begin{eqnarray*}
E_1 & = & E_{\Delta_L},\\
C''_{\delta} & = & C'_{\delta} + 1 + 4\max_{g\in F'_{\delta}} \Vert\mu(g)\Vert,\\
x & = & \sum_{i=0}^n \mu(\gamma_i),\\
x' & = & \mu(\varphi(\gamma)):
\end{eqnarray*}
we obtain that if $\delta$ is small enough, then
\begin{equation}\label{mu(varphi(gamma)) et somme des mu(gamma_i)}
\Big\Vert\mu(\varphi(\gamma)) - \sum_{i=0}^n \mu(\gamma_i)\Big\Vert \,\leq\, \frac{\varepsilon}{4}\ \Big\Vert\sum_{i=0}^n \mu(\gamma_i)\Big\Vert + \frac{4C''_{\delta}}{\delta}
\end{equation}
for all $\varphi\in\mathcal{U}'_{\delta}$.
Now Conditions (1) and~(3), together with the triangular inequality and~(\ref{norme de la somme des mu(gamma_i)}), imply that
\begin{eqnarray*}
\Big\Vert\Big(\sum_{i=0}^n \mu(\varphi(\gamma_i))\Big) - \Big(\sum_{i=0}^n \mu(\gamma_i)\Big)\Big\Vert & \leq & \delta\,\Big(\sum_{i=0}^n \Vert\mu(\gamma_i)\Vert\Big) + 1,\\
& \leq & \delta\,\Big\Vert\sum_{i=0}^n \mu(\gamma_i)\Big\Vert + 2\delta\,\max_{g\in F'_{\delta}} \Vert\mu(g)\Vert + 1.
\end{eqnarray*}
Therefore, if $\delta$ is small enough, then
$$\Big\Vert\mu(\varphi(\gamma)) - \sum_{i=0}^n \mu(\varphi(\gamma_i))\Big\Vert \,\leq\, \frac{\varepsilon}{2}\ \Big\Vert\sum_{i=0}^n \mu(\gamma_i)\Big\Vert + C'''_{\delta}$$
for all~$\varphi\in\mathcal{U}'_{\delta}$, where
$$C'''_{\delta} = \frac{4C''_{\delta}}{\delta} + 2\delta\,\max_{g\in F'_{\delta}} \Vert\mu(g)\Vert + 1.$$
In particular, taking $\varphi$ to be the natural inclusion of~$\Gamma$ in~$G$, we get
\begin{equation}\label{mu(gamma) et somme des mu(gamma_i)}
\Big\Vert\sum_{i=0}^n \mu(\gamma_i)\Big\Vert \,\leq\, \frac{1}{1-\frac{\varepsilon}{2}}\,\big(\Vert\mu(\gamma)\Vert + C'''_{\delta}\big) \,\leq\, 2\,\Vert\mu(\gamma)\Vert + 2C'''_{\delta},
\end{equation}
hence
$$\Big\Vert\mu(\varphi(\gamma)) - \sum_{i=0}^n \mu(\varphi(\gamma_i))\Big\Vert \leq \varepsilon\,\Vert\mu(\gamma)\Vert + 3C'''_{\delta}$$
for all~$\varphi\in\mathcal{U}'_{\delta}$ whenever $\delta$ is small enough.
Finally, Condition~(1), together with (\ref{norme de la somme des mu(gamma_i)}) and~(\ref{mu(gamma) et somme des mu(gamma_i)}), implies that
$$n \,\leq\, 2\delta\,\Vert\mu(\gamma)\Vert + 3C'''_{\delta}.$$
Thus the triple $(F_{\varepsilon},\mathcal{U}_{\varepsilon},C_{\varepsilon})=(F'_{\delta},\mathcal{U}'_{\delta},3C'''_{\delta})$ satisfies the conclusions of Proposition~\ref{mu d'un produit et somme des mu} for $\delta$ small enough.
\end{proof}

\subsection{Properness and deformation}\label{Proprete and deformations}

Let us now briefly explain how to deduce Theorems~\ref{proprete, groupes de Lie} and~\ref{proprete, groupes algebriques} from Theorem~\ref{varphi ne change pas beaucoup mu}.

Theorem~\ref{proprete, groupes algebriques} follows from Theorem~\ref{varphi ne change pas beaucoup mu} and from the \emph{properness criterion} of Benoist (\cite{ben96}, Cor.~5.2) and Kobayashi (\cite{kob96}, Th.~1.1).
Under the assumptions of Theorem~\ref{proprete, groupes algebriques}, this criterion states that a subgroup~$\Gamma$ of~$G$ acts properly on~$G/H$ if and only if the set $\mu(\Gamma)\cap(\mu(H)+\mathcal{C})$ is bounded for any compact subset~$\mathcal{C}$ of~$E$.
This condition means that the set~$\mu(\Gamma)$ ``gets away from~$\mu(H)$ at infinity''.

\begin{proof}[Proof of Theorem~\ref{proprete, groupes algebriques}]
We may assume that $\mathbf{G}$, $\mathbf{H}$, and $\mathbf{L}$ are all connected.
Let $G=KA^+K$ or $G=KZ^+K$ be a Cartan decomposition of~$G$, and let $\mu : G\rightarrow E^+$ be the corresponding Cartan projection.
Endow~$E$ with a \linebreak $W$-invariant norm~$\Vert\cdot\Vert$ as in Section~\ref{Preliminaires}.
By Remark~\ref{decompositions de Cartan compatibles}, we may assume that $L$ admits a Cartan decomposition $L=K_LA_L^+K_L$ or $L=K_LZ_L^+K_L$ with $K_L\subset K$ and $\mathbf{A_L}\subset\mathbf{A}$.
If $\mu_L : L\rightarrow E_L^+$ denotes the corresponding Cartan projection, then $E_L$ is naturally seen as a line in~$E$ and
$$\mu(\ell) = E^+\cap (W\cdot\mu_L(\ell))$$
for all~$\ell\in L$.
Thus $\mu(L)$ is contained in the union~$U_L$ of two half-lines of~$E^+$.
Using Remark~\ref{decompositions de Cartan compatibles} again, there is an element~$g\in G$ such that $gHg^{-1}$ admits a Cartan decomposition $gHg^{-1}=K_HA_H^+K_H$ or $gHg^{-1}=K_HZ_H^+K_H$ with $K_H\subset K$ and $\mathbf{A_H}\subset\mathbf{A}$.
The set $\mu(gHg^{-1})$ is contained in a finite union~$U_H$ of subspaces of~$E$ intersected with~$E^+$, parametrized by the Weyl group~$W$.
By~(\ref{inegalite fine pour mu}), the Hausdorff distance between $\mu(gHg^{-1})$ and~$\mu(H)$ is $\leq 2\,\Vert\mu(g)\Vert$.
Therefore $U_L\cap U_H=\{ 0\} $ by the properness criterion, and there is a constant $\varepsilon>0$ such that
$$d(\mu(\ell),\mu(H))\geq 2\varepsilon\,\Vert\mu(\ell)\Vert-2\,\Vert\mu(g)\Vert$$
for all~$\ell\in L$.
By Theorem~\ref{varphi ne change pas beaucoup mu}, there is a neighborhood $\mathcal{U}_{\varepsilon}\subset\Hom(\Gamma,G)$ of the natural inclusion and a constant $C_{\varepsilon}\geq 0$ such that
$$\Vert\mu(\varphi(\gamma))-\mu(\gamma)\Vert \leq \varepsilon\,\Vert\mu(\gamma)\Vert + C_{\varepsilon}$$
for all $\varphi\in\mathcal{U}_{\varepsilon}$ and $\gamma\in\Gamma$.
Fix $\varphi\in\mathcal{U}_{\varepsilon}$.
For all~$\gamma\in\Gamma$,
\begin{eqnarray*}
d(\mu(\varphi(\gamma)),\mu(H)) & \geq & d(\mu(\gamma),\mu(H)) - \Vert\mu(\varphi(\gamma))-\mu(\gamma)\Vert\\
& \geq & \varepsilon\,\Vert\mu(\gamma)\Vert - C_{\varepsilon} - 2\,\Vert\mu(g)\Vert.
\end{eqnarray*}
Therefore, using the fact that $\Gamma$ is discrete in~$G$ and $\mu$ is a proper map, we get that $\mu(\varphi(\Gamma))\cap(\mu(H)+\mathcal{C})$ is finite for any compact subset~$\mathcal{C}$ of~$E$.
By the properness criterion, this implies that $\varphi(\Gamma)$ acts properly on~$G/H$.
It also implies that $\varphi(\Gamma)$ is discrete in~$G$ and that the kernel of~$\varphi$ is finite.
Since $\Gamma$ is torsion-free, $\varphi$ is injective.
\end{proof}

\begin{proof}[Proof of Theorem~\ref{proprete, groupes de Lie}]
By Theorem~\ref{proprete, groupes algebriques}, there is a neighborhood $\mathcal{U}\subset\linebreak\Hom(\Gamma,G)$ of the natural inclusion such that any $\varphi\in\mathcal{U}$ is injective and $\varphi(\Gamma)$ is discrete in~$G$, acting properly discontinuously on~$G/H$.
Since $\varphi\in\mathcal{U}$ is injective, $\varphi(\Gamma)$ has the same cohomological dimension as~$\Gamma$.
We conclude using the fact, due to Kobayashi (\cite{kob89}, Cor.~5.5), that a torsion-free discrete subgroup of~$G$ acts cocompactly sur~$G/H$ if and only if its cohomological dimension is $d(G)-d(H)$, where $d(G)$ (resp.~$d(H)$) denotes the dimension of the symmetric space of~$G$ (resp.\ of~$H$).
\end{proof}

\section{Application to the compact quotients of $\SO(2n,2)/\U(n,1)$}\label{quotients compacts Zariski-denses}

Fix an integer~$n\geq 1$.
Note that $\U(n,1)$ naturally embeds in~$\SO(2n,2)$ by identifying the Hermitian form $|z_1|^2+\ldots+|z_n|^2-|z_{n+1}|^2$ on~$\C^{n+1}$ with the quadratic form $x_1^2+\ldots+x_{2n}^2-x_{2n+1}^2-x_{2n+2}^2$ on~$\R^{2n+2}$.
As Kulkarni~\cite{kul} pointed out, $\U(n,1)$, seen as a subgroup of~$\SO(2n,2)$, acts transitively on the anti-de Sitter space
\begin{eqnarray*}
\mathrm{AdS}^{2n+1} & = & \big\{ (x_1,\ldots,x_{2n+2})\in\R^{2n+2},\ x_1^2 + \ldots + x_{2n}^2 - x_{2n+1}^2 - x_{2n+2}^2 = -1\big\} \\
& \simeq & \SO(2n,2)/\SO(2n,1).
\end{eqnarray*}
The stabilizer of $(0,\ldots,0,1)$ is the compact subgroup~$\U(n)$, hence $\mathrm{AdS}^{2n+1}$ identifies with $\U(n,1)/\U(n)$ and the action of~$\U(n,1)$ on $\SO(2n,2)/\SO(2n,1)$ is proper.
By duality, the action of $\SO(2n,1)$ on $\SO(2n,2)/\U(n,1)$ is proper and transitive.
In particular, any uniform lattice~$\Gamma$ of~$\SO(2n,1)$ provides a standard compact quotient $\Gamma\backslash\SO(2n,2)/\U(n,1)$ of $\SO(2n,2)/\U(n,1)$.

Corollary~\ref{quotients compacts de SO(2n,2)/SU(n,1)} follows from Theorem~\ref{proprete, groupes de Lie} and from the existence of Zariski-dense deformations in~$\SO(m,2)$ of certain uniform lattices of~$\SO(m,1)$.
Such deformations can be obtained by a bending construction due to Johnson and Millson.
This construction is presented in~\cite{jm} for deformations in \linebreak $\SO(m+1,1)$ or in~$\PGL_{m+1}(\R)$, but not in~$\SO(m,2)$.
For the reader's convenience, we shall describe Johnson and Millson's construction in the latter case, and check that the deformations obtained is this way are indeed Zariski-dense in~$\SO(m,2)$.

From now on we use Gothic letters to denote the Lie algebras of real Lie groups (\textit{e.g.} $\g$ for~$G$).

\subsection{Uniform arithmetic lattices of~$\SO(m,1)$}

Fix~$m\geq 2$.
The uniform lattices of~$\SO(m,1)$ considered by Johnson and Millson are obtained in the following classical way.
Fix a square-free integer~$r\geq 2$ and identify $\SO(m,1)$ with the special orthogonal group of the quadratic form
$$x_1^2+\ldots+x_m^2-\sqrt{r}x_{m+1}^2$$
on~$\R^{m+1}$.
Let $\mathcal{O}_r$ denote the ring of integers of the quadratic field~$\Q(\sqrt{r})$.
The group $\Gamma=\SO(m,1)\cap\M_{m+1}(\mathcal{O}_r)$ is a uniform lattice in~$\SO(m,1)$ (see \cite{bor63} for instance).
For any ideal~$I$ of~$\mathcal{O}_r$, the congruence subgroup $\Gamma\cap (1+\M_{m+1}(I))$ has finite index in~$\Gamma$, hence is a uniform lattice in~$\SO(m,1)$.
By~\cite{mr}, after replacing~$\Gamma$ by such a congruence subgroup, we may assume that it is torsion-free.
Then $M=\Gamma\backslash\mathbb{H}^m$ is a $m$-dimensional compact hyperbolic manifold whose fundamental group identifies with~$\Gamma$.
By~\cite{jm}, Lem.~7.1 \& Th.~7.2, after possibly replacing~$\Gamma$ again by some congruence subgroup, we may assume that $N=\Gamma_0\backslash\HH^{m-1}$ is a connected, orientable, totally geodesic hypersurface of~$M$, where
$$\Gamma_0=\Gamma\cap\SO(m-1,1)$$
and where
$$\HH^{m-1} \simeq \big\{ (x_2,\ldots,x_{m+1})\in\R^m,\ x_2^2+\ldots+x_m^2-\sqrt{r}x_{m+1}^2=-1\ \mathrm{and}\ x_{m+1}>0\big\} $$
is embedded in
$$\HH^m \simeq \{ (x_1,\ldots,x_{m+1})\in\R^{m+1},\ x_1^2+\ldots+x_m^2-\sqrt{r}x_{m+1}^2=-1\ \mathrm{and}\ x_{m+1}>0\} $$
in the natural way.
Since the centralizer of~$\Gamma_0$ in~$\SO(m,2)$ contains a subgroup isomorphic to~$\SO(1,1)\simeq\R^{\ast}$, the idea of the bending construction is to deform~$\Gamma$ ``along this centralizer'', as we shall now explain.

\subsection{Deformations in the separating case}

Assume that $N$ separates~$M$ into two components $M_1$ and~$M_2$, and let $\Gamma_1$ (resp.~$\Gamma_2$) denote the fundamental group of~$M_1$ (resp.\ of~$M_2$).
By van Kampen's theorem, $\Gamma$ is the amalgamated product $\Gamma_1\ast_{\Gamma_0}\Gamma_2$.
Fix an element $Y\in\so(m,2)\smallsetminus\so(m,1)$ that belongs to the Lie algebra of the centralizer of~$\Gamma_0$ in~$\SO(m,2)$.
Following Johnson and Millson, we consider the deformations of~$\Gamma$ in~$\SO(m,2)$ that are given, for $t\in\R$, by
\begin{equation*}
  \varphi_t(\gamma) =
     \begin{cases}
        \ \ \,\,\,\,\,\gamma & \text{for $\gamma\in\Gamma_1$,}\\
        e^{tY}\gamma e^{-tY} & \text{for $\gamma\in\Gamma_2$.}
     \end{cases}
\end{equation*}
Note that $\varphi_t : \Gamma\rightarrow\SO(m,2)$ is well defined since $e^{tY}$ centralizes~$\Gamma_0$.
Moreover, it is injective and $\varphi_t(\Gamma)$ is discrete in~$\SO(m,2)$.
We now check Zariski-density.

\begin{lem}\label{lemme cas separant}
For~$t\neq 0$ small enough, $\varphi_t(\Gamma)$ is Zariski-dense in~$\SO(m,2)$.
\end{lem}

We need the following remark.

\begin{rema}\label{remarque algebres de Lie}
For~$m\geq 2$, the only Lie subalgebra of~$\so(m,2)$ that strictly contains~$\so(m,1)$ is~$\so(m,2)$.
\end{rema}

\noindent
Indeed, $\so(m,2)$ decomposes uniquely into a direct sum $\so(m,1)\oplus\nolinebreak\R^{m+1}$ of irreducible $\SO(m,1)$-modules, where $\SO(m,1)$ acts on~$\so(m,1)$ (resp.\ on~$\R^{m+1}$) by the adjoint (resp.\ natural) action.

\begin{proof}[Proof of Lemma~\ref{lemme cas separant}]
Recall that~$\SO(m,2)$ is Zariski-connected.
Therefore, in order to prove that $\varphi_t(\Gamma)$ is Zariski-dense in~$\SO(m,2)$, it is sufficient to prove that the Lie algebra of~$\overline{\varphi_t(\Gamma)}$ is~$\so(m,2)$, where $\overline{\varphi_t(\Gamma)}$ denotes the Zariski closure of~$\varphi_t(\Gamma)$ in~$\SO(m,2)$.

By~\cite{jm}, Lem.~5.9, the groups $\Gamma_1$ and~$\Gamma_2$ are Zariski-dense in~$\SO(m,1)$.
By \cite{jm}, Cor.~5.3, and \cite{ser77}, \S~I.5.2, Cor.~1, they naturally embed in~$\Gamma$.
Therefore $\overline{\varphi_t(\Gamma)}$ contains both $\SO(m,1)$ and~$e^{tY}\SO(m,1) e^{-tY}$, and the Lie algebra of~$\overline{\varphi_t(\Gamma)}$ contains both $\so(m,1)$ and the Lie algebra of $e^{tY}\SO(m,1) e^{-tY}$.
By Remark~\ref{remarque algebres de Lie}, in order to prove that $\varphi_t(\Gamma)$ is Zariski-dense in~$\SO(m,2)$, it is sufficient to prove that the Lie algebra of $e^{tY}\SO(m,1) e^{-tY}$ is not~$\so(m,1)$.

But if the Lie algebra of $e^{tY}\SO(m,1) e^{-tY}$ were~$\so(m,1)$, then we would have $e^{tY}\SO(m,1)^{\circ}e^{tY}=\SO(m,1)^{\circ}$, \textit{i.e.}, $e^{tY}$ would belong to the normalizer $N_{\SO(m,2)}(\SO(m,1)^{\circ})$ of the identity component~$\SO(m,1)^{\circ}$ of~$\SO(m,1)$.
Recall that the exponential map induces a diffeomorphism between a neighborhood~$\mathcal{U}$ of~$0$ in~$\so(m,2)$ and a neighborhood~$\mathcal{V}$ of~$1$ in~$\SO(m,2)$, which itself induces a one-to-one correspondence between $\mathcal{U}\cap\n_{\so(m,2)}(\so(m,1))$ and $\mathcal{V}\cap N_{\SO(m,2)}(\SO(m,1)^{\circ})$.
Therefore, if we had $e^{tY}\in N_{\SO(m,2)}(\SO(m,1)^{\circ})$ for some $t\neq 0$ small enough, then we would have
$$Y \in \n_{\so(m,2)}(\so(m,1)) = \{ X\in\so(m,2),\ \ad(X)(\so(m,1))=\so(m,1)\} .$$
But Remark~\ref{remarque algebres de Lie} implies that $\n_{\so(m,2)}(\so(m,1))$ is equal to~$\so(m,1)$, since it contains~$\so(m,1)$ and is different from~$\so(m,2)$.
Thus we would have $Y\in\so(m,1)$, which would contradict our choice of~$Y$.
\end{proof}

\subsection{Deformations in the nonseparating case}

We now assume that $S=M\smallsetminus N$ is connected.
Let $j_1:\Gamma_0\rightarrow\pi_1(S)$ and $j_2:\Gamma_0\rightarrow\pi_1(S)$ denote the inclusions in~$\pi_1(S)$ of the fundamental groups of the two sides of~$N$.
The group~$\Gamma$ is a HNN extension of~$\pi_1(S)$, \textit{i.e.}, it is generated by~$\pi_1(S)$ and by some element $\nu\in\Gamma$ such that
$$\nu\,j_1(\gamma)\,\nu^{-1} = j_2(\gamma)$$
for all $\gamma\in\Gamma_0$.
Fix an element $Y\in\so(m,2)\smallsetminus\so(m,1)$ that belongs to the Lie algebra of the centralizer of~$j_1(\Gamma_0)$ in~$\SO(m,2)$.
Following Johnson and Millson, we consider the deformations of~$\Gamma$ in~$\SO(m,2)$ that are given, for $t\in\R$, by
$$\left \{
\begin{array}{c @{\ =\ } l l}
    \varphi_t(\gamma) & \ \,\gamma & \mathrm{for}\ \gamma\in\pi_1(S),\\
    \varphi_t(\nu) & \nu e^{tY}.
\end{array}
\right.$$
Note that $\varphi_t : \Gamma\rightarrow\SO(m,2)$ is well defined since $e^{tY}$ centralizes~$j_1(\Gamma_0)$.
Moreover, it is injective and $\varphi_t(\Gamma)$ is discrete in~$\SO(m,2)$.

\begin{lem}
For $t\neq 0$ small enough, $\varphi_t(\Gamma)$ is Zariski-dense in~$\SO(m,2)$.
\end{lem}

\begin{proof}
Let $\overline{\varphi_t(\Gamma)}$ denote the Zariski closure of~$\varphi_t(\Gamma)$ in~$\SO(m,2)$.
By~\cite{jm}, Lem.~5.9, the group~$\pi_1(S)$ is Zariski-dense in~$\SO(m,1)$, hence $\overline{\varphi_t(\Gamma)}$ contains both~$\SO(m,1)$ and~$\nu e^{tY}$.
But $\nu\in\SO(m,1)$, hence $e^{tY}\in\overline{\varphi_t(\Gamma)}$.
Therefore $\overline{\varphi_t(\Gamma)}$ contains both $\SO(m,1)$ and $e^{tY}\SO(m,1)e^{-tY}$, and we may conclude as in the proof of Lemma~\ref{lemme cas separant}.
\end{proof}


\vspace{0.5cm}


\begin{thebibliography}{[BoT]}

\medskip

\bibitem[Bas]{bas}
\textsc{H. Bass}, \textit{Covering theory for graphs of groups}, J. Pure Appl. Algebra~89 (1993), p.~3--47.

\medskip

\bibitem[Ben]{ben96}
\textsc{Y. Benoist}, \textit{Actions propres sur les espaces homog\`enes r\'eductifs}, Ann. Math.~144 (1996), p.~315--347, announced in C. R. Acad. Sci. Paris, Ser.~I, Math.~319 (1994), p.~937--940.

\medskip

\bibitem[Bo1]{bor63}
\textsc{A. Borel}, \textit{Compact Clifford-Klein forms of symmetric spaces}, Topology~2 (1963), p.~111--122.

\medskip

\bibitem[Bo2]{bor91}
\textsc{A. Borel}, \textit{Linear algebraic groups}, Second edition, Graduate Texts in Mathematics~126, Springer-Verlag, New York, 1991.

\medskip

\bibitem[BoT]{bot}
\textsc{A. Borel, J. Tits}, \textit{Groupes r\'eductifs}, Publ. Math. Inst. Hautes \'Etudes Sci.~27 (1965), p.~55--150.

\medskip

\bibitem[Bou]{bou}
\textsc{M. Bourdon}, \textit{Structure conforme au bord et flot g\'eod\'esique d'un $CAT(-1)$-espace}, Enseign. Math.~41 (1995), p.~63--102.

\medskip

\bibitem[BH]{brh}
\textsc{M. R. Bridson, A. Haefliger}, \textit{Metric spaces of non-positive curvature}, Grundlehren der mathematischen Wissenschaften~319, Springer-Verlag, Berlin, 1999.

\medskip

\bibitem[BT1]{bt1}
\textsc{F. Bruhat, J. Tits}, \textit{Groupes r\'eductifs sur un corps local : I. Donn\'ees radicielles valu\'ees}, Publ. Math. Inst. Hautes \'Etudes Sci.~41 (1972), p.~5--251.

\medskip

\bibitem[BT2]{bt2}
\textsc{F. Bruhat, J. Tits}, \textit{Groupes r\'eductifs sur un corps local : II. Sch\'emas en groupes. Existence d'une donn\'ee radicielle valu\'ee}, Publ. Math. Inst. Hautes \'Etudes Sci.~60 (1984), p.~5--184.

\medskip

\bibitem[Che]{che}
\textsc{C. Chevalley}, \textit{Th\'eorie des groupes de Lie : II. Groupes alg\'ebriques}, Actualit\'es Sci. Ind. 1152, Hermann \& Cie, Paris, 1951.

\medskip

\bibitem[Ghy]{ghy}
\textsc{\'E. Ghys}, \textit{D\'eformations des structures complexes sur les espaces homog\`enes de~$\mathrm{SL}_2(\mathbb{C})$}, J. Reine Angew. Math.~468 (1995), p.~113--138.

\medskip

\bibitem[Gol]{gol}
\textsc{W. M. Goldman}, \textit{Nonstandard Lorentz space forms}, J. Differ. Geom.~21 (1985), p.~301--308.

\medskip

\bibitem[Gui]{gui}
\textsc{O. Guichard}, \textit{Groupes plong\'es quasi-isom\'etriquement dans un groupe de Lie}, Math. Ann.~330 (2004), p.~331--351.

\medskip

\bibitem[Hel]{hel}
\textsc{S. Helgason}, \textit{Differential geometry, Lie groups, and symmetric spaces}, Corrected reprint of the 1978 original, Graduate Studies in Mathematics 34, American Mathematical Society, Providence, RI, 2001.

\medskip

\bibitem[JM]{jm}
\textsc{D. Johnson, J. J. Millson}, \textit{Deformation spaces associated to compact hyperbolic manifolds}, in \textit{Discrete groups in geometry and analysis}, p.~48--106, Prog. Math.~67, Birkhäuser, Boston, MA, 1987.

\medskip

\bibitem[Kar]{kar}
\textsc{F. I. Karpelevich}, \textit{Surfaces of transitivity of semisimple groups of motions of a symmetric space}, Soviet Math. Dokl. 93 (1953), p.~401--404.

\medskip

\bibitem[Kas]{kas08}
\textsc{F. Kassel}, \textit{Proper actions on corank-one reductive homogeneous spaces}, J.~Lie Theory~18 (2008), p.~961--978.

\medskip

\bibitem[Kli]{kli}
\textsc{B. Klingler}, \textit{Compl\'etude des vari\'et\'es lorentziennes \`a courbure constante}, Math. Ann.~306 (1996), p.~353--370.

\medskip

\bibitem[Ko1]{kob89}
\textsc{T. Kobayashi}, \textit{Proper action on a homogeneous space of reductive type}, Math. Ann.~285 (1989), p.~249--263.

\medskip

\bibitem[Ko2]{kob96}
\textsc{T. Kobayashi}, \textit{Criterion for proper actions on homogeneous spaces of reductive groups}, J. Lie Theory~6 (1996), p.~147--163.

\medskip

\bibitem[Ko3]{kob98}
\textsc{T. Kobayashi}, \textit{Deformation of compact Clifford-Klein forms of indefinite-Riemannian homogeneous manifolds}, Math. Ann.~310 (1998), p.~394--408.

\medskip

\bibitem[KY]{ky}
\textsc{T. Kobayashi, T. Yoshino}, \textit{Compact Clifford-Klein forms of symmetric spaces --- revisited}, Pure and Applied Mathematics Quaterly~1 (2005), p.~591--653.

\medskip

\bibitem[Kul]{kul}
\textsc{R. S. Kulkarni}, \textit{Proper actions and pseudo-Riemannian space forms}, Adv. Math.~40 (1981), p.~10--51.

\medskip

\bibitem[Lan]{lan}
\textsc{E. Landvogt}, \textit{Some functorial properties of the Bruhat-Tits building}, J. Reine Angew. Math. 518 (2000), p.~213--241.

\medskip

\bibitem[Mar]{mar91}
\textsc{G. A. Margulis}, \textit{Discrete subgroups of semisimple Lie groups}, Ergebnisse der Mathematik und ihrer Grenzgebiete (3) 17, Springer-Verlag, Berlin, 1991.

\medskip

\bibitem[MR]{mr}
\textsc{J. J. Millson, M. S. Raghunathan}, \textit{Geometric construction of cohomology for arithmetic groups~I}, in \textit{Geometry and analysis}, p.~103--123, Indian Acad. Sci., Bangalore, 1980.

\medskip

\bibitem[Mos]{mos55}
\textsc{G. D. Mostow}, \textit{Some new decomposition theorems for semi-simple groups}, Mem. Amer. Math. Soc.~14 (1955), p.~31--54.

\medskip

\bibitem[Qui]{qui}
\textsc{J.-F. Quint}, \textit{Cônes limites des sous-groupes discrets des groupes r\'eductifs sur un corps local}, Transform. Groups~7 (2002), p.~247--266.

\medskip

\bibitem[Rag]{rag}
\textsc{M. S. Raghunathan}, \textit{On the first cohomology of discrete subgroups of semisimple Lie groups}, Amer. J. Math. 87 (1965), p.~103--139.

\medskip

\bibitem[Rou]{rou}
\textsc{G. Rousseau}, \textit{Euclidean buildings}, in ``Nonpositive curvature geometries, discrete groups and rigidity'', Proceedings of the 2004 Grenoble summer school, S\'eminaires et Congr\`es~18, Soci\'et\'e Math\'ematique de France, Paris, 2008.

\medskip

\bibitem[Sal]{sal00}
\textsc{F. Salein}, \textit{Vari\'et\'es anti-de Sitter de dimension~3 exotiques}, Ann. Inst. Fourier~50 (2000), p.~257--284.

\medskip

\bibitem[Sel]{sel}
\textsc{A. Selberg}, \textit{On discontinuous groups in higher-dimensional symmetric spaces}, in ``Collected papers'', vol.~1, p.~475--492, Springer-Verlag, Berlin,~1989.

\medskip

\bibitem[Ser]{ser77}
\textsc{J.-P. Serre}, \textit{Arbres, amalgames, $\mathrm{SL}(2)$}, Ast\'erisque~46, Soci\'et\'e Math\'ematique de France, Paris, 1977.

\medskip

\bibitem[Tit]{tit71}
\textsc{J. Tits}, \textit{Repr\'esentations lin\'eaires irr\'eductibles d'un groupe r\'eductif sur un corps quelconque}, J. Reine Angew. Math.~247 (1971), p.~196--220.

\medskip

\bibitem[Wei]{wei64}
\textsc{A. Weil}, \textit{Remarks on the cohomology of groups}, Ann. Math.~80 (1964), p.~149--157.

\medskip

\bibitem[Zeg]{zeg}
\textsc{A. Zeghib}, \textit{On closed anti-de Sitter spacetimes}, Math. Ann.~310 (1998), p.~695--716.

\end{thebibliography}
\end{document}